\documentclass[11pt]{article}
\usepackage{graphicx}
\usepackage{amsfonts}
\usepackage{amsmath}
\usepackage{amsthm}
\usepackage{amssymb}
\usepackage{fancyhdr}
\usepackage{indentfirst}
\usepackage{titlesec}
\usepackage{hyperref}
\usepackage{abstract}

\newcommand{\al}{\alpha}
\newcommand{\be}{\beta}
\newcommand{\V}{\mathcal{V}}
\newcommand{\R}{\mathbb{R}}
\newcommand{\N}{\mathbb{N}}
\newcommand{\mH}{\mathcal{H}}
\newcommand{\F}{\mathcal{F}}

\newcommand{\B}{\mathcal{B}}
\newcommand{\VC}{\mathcal{VC}}
\newcommand{\MVC}{\mathcal{MVC}}

\newcommand{\lan}{\langle}
\newcommand{\ran}{\rangle}
\newcommand{\la}{\lambda}
\newcommand{\La}{\Lambda}
\newcommand{\si}{\sigma}
\newcommand{\Si}{\Sigma}
\newcommand{\de}{\delta}
\newcommand{\ep}{\epsilon}
\newcommand{\pr}{\prime}
\newcommand{\Om}{\Omega}
\newcommand{\Ga}{\Gamma}
\newcommand{\ga}{\gamma}
\newcommand{\ka}{\kappa}
\newcommand{\lap}{\triangle}
\newcommand{\ti}{\tilde}

\newcommand{\mL}{\mathcal{L}}

\newcommand{\mZ}{\mathbb{Z}}

\newtheorem{theorem}{Theorem}[section]
\newtheorem{proposition}[theorem]{Propostion}

\theoremstyle{remark}
\newtheorem{remark}[theorem]{Remark}

\theoremstyle{definition}
\newtheorem{definition}[theorem]{Definition}

\theoremstyle{plain}
\newtheorem{lemma}[theorem]{Lemma}

\numberwithin{equation}{section}

\begin{document}




\pagestyle{headings}
\renewcommand{\headrulewidth}{0.4pt}

\title{\textbf{Mass angular momentum inequality for axisymmetric vacuum data with small trace}}
\author{Xin Zhou\\}
\date{}
\maketitle

\pdfbookmark[0]{Mass angular momentum inequality for axisymmetric vacuum data with small trace}{}

\renewcommand{\abstractname}{}    
\renewcommand{\absnamepos}{empty} 
\begin{abstract}
\noindent\textbf{Abstract:} In this paper, we proved the mass angular momentum inequality\cite{D1}\cite{ChrusLiWe}\cite{SZ} for axisymmetric, asymptotically flat, vacuum constraint data sets with small trace. Given an initial data set with small trace, we construct a boost evolution spacetime of the Einstein vacuum equations as \cite{ChOM}. Then a perturbation method is used to solve the maximal surface equation in the spacetime under certain growing condition at infinity. When the initial data set is axisymmetric, we get an axisymmetric maximal graph with the same ADM mass and angular momentum as the given data. The inequality follows from the known results\cite{D1}\cite{ChrusLiWe}\cite{SZ} about the maximal graph.
\end{abstract}



\section{Introduction}

Based on the gravitational collapse pictures \cite{D}, it is conjectured that the angular momentum should be bounded by the mass for physically reasonable solutions of the Einstein equations. It is true for Kerr black hole solutions which are stationary. For dynamical, axisymmetric solutions some progresses have been made over the past few years. Dain \cite{D1} first proved such an inequality for Brill data(See Definition 2.1 of \cite{D1}), which is a class of specialized axisymmetric, maximal, asymptotically flat vacuum data. Later,  Chru\'sciel, Li and Weinstein \cite{Chrus1}\cite{ChrusLiWe} generalized it to a class of axisymmetric, maximal data admitting an Ernst potential with positive mass density and certain asymptotically flatness conditions. Recently R. Schoen and the author \cite{SZ} gave a simplified proof for more general asymptotic conditions and an $L^{6}$ norm bound.

All the existing results require the solutions to be maximal, which restricts the data to be a special time-slice in a spacetime. However it should be unnecessary according to the gravitational collapse pictures\footnote{The axisymmetric condition is indeed necessary, since otherwise vacuum counterexamples were constructed by Huang, Schoen and Wang \cite{HSW}}. It is natural and interesting to study the non-maximal case. In this paper we will prove the mass angular momentum inequality for non-maximal vacuum data with small trace by exploring the Einstein equations and a perturbation method. Using notations in Section \ref{ideas and main results}, our main theorem is


\begin{theorem}\label{mass angular momentum corollary}\emph{(\textbf{Main Theorem 1})}
Suppose $(\Si, e)$ is a simply connected 3-manifold, which is Euclidean at infinity with two ends and axisymmetric in the sense of Definition \ref{definition of axisymmetric}. Given an asymptotically flat, axisymmetric vacuum data $(g, k)\in\VC^{a}_{s+2, \de+\frac{1}{2}}(\Si)$(see Definition \ref{definition of axisymmetric2}) with $s\in\N$, $s\geq 7$, $\de\in\R$, $-\frac{3}{2}<\delta<-1$, if $\|tr_{g}k\|_{H_{s-2, \de+\frac{3}{2}}(\Si)}\leq\ep$ with $\ep$ given in Theorem \ref{main theorem}, we have
\begin{equation}
m\geq\sqrt{|J|},
\end{equation}
where $m$ and $J$ are the ADM mass (\ref{ADM mass}) and angular momentum (\ref{angular momentum}) of $(\Si, g, k)$ respectively.
\end{theorem}

Our method comes from a question suggested by R. Schoen:
\vspace{-6pt}
\begin{itemize}
\addtolength{\itemsep}{-0.7em}
\item[$\textbf{(Q)}:$] \emph{Is there a canonical way to deform a non-maximal, axisymmetric, vacuum data to a unique maximal, vacuum data with the same physical quantities, i.e. the mass and angular momentum, which also preserves the axially symmetry?}
\end{itemize}
\vspace{-6pt}
A definite answer of the above question will imply the mass angular momentum inequality in the non-maximal case. In fact, there are already some works about the deformation of vacuum constraint equations \cite{BaI}\cite{CS}. But it is hard to maintain the symmetries and physical quantities at the same time. So the main difficulty is to maintain the symmetries and the physical quantities simultaneously when deforming the vacuum constraint equations. We overcome this difficulty by using certain conversation laws of the Einstein equations.

\subsection{General Relativity backgrounds}

In Einstein's theory for General Relativity\footnote{We refer to \cite{W} for all the concepts.}, we use $(\mathcal{V}^{3, 1}, \gamma)$ to denote a spacetime, where $\V^{3, 1}$ is a 4-dimensional oriented smooth manifold, and $\gamma$ is a Lorentzian metric of signature $(3, 1)$. The Einstein equation, which predicts the evolution of the spacetime, is given by
\begin{equation}\label{einstein equation 1}
Ric_{\gamma}-\frac{1}{2}R_{\gamma}\gamma=8\pi T,
\end{equation}
where $Ric_{\gamma}$ is the Ricci curvature of $\gamma$, and $R_{\gamma}$ the scalar curvature of $\gamma$. $T$ is the stress-energy tensor. In the vacuum case, $T\equiv 0$, so the \emph{Einstein vacuum equation}, abbreviated as (\textbf{EVE}) in the following, reduces to
\begin{equation}\label{EVE}
Ric_{\gamma}=0.
\end{equation}

A \emph{vacuum constraint initial data set} or abbreviated as \emph{vacuum data} for the Einstein vacuum equations is a triple $(\Sigma, g, k)$, where $\Sigma$ is a connected complete 3-dimensional manifold, $g$ a Riemmanian metric, and $k$ a symmetric two tensor on $\Sigma$, satisfying the \emph{vacuum constrain equations}, abbreviated as (\textbf{VCE}),
\begin{equation}\label{VCE}
\left.\bigg\{ \begin{array}{ll}
R_{g}-|k|^{2}_{g}+(tr_{g}k)^{2}=0,\\
div_{g}(k-(tr_{g}k)g)=0.
\end{array} \right.
\end{equation}
By the famous initial value formulation for the Einstein equations by Y. Choquet-Bruhat in 1952(see \cite{Ch2}\cite{W}), we can always think the vacuum data $(\Si, g, k)$ as been embedded in some spacetime $(\V, \ga)$ satisfying (EVE), where $g$ is the restriction of $\ga$ to $\Si$, and $k$ is the second fundamental form of the embedding.

$(\Si, e)$ is called \emph{Euclidean at infinity}, where $e$ is a Riemannian metric on $\Si$, if there is a compact subset $\Si_{int}\subset\Si$, such that the compliment $\Si_{ext}=\Si\setminus\Si_{int}$ is a disjoint union of finitely many open sets $\Si_{ext}=\cup_{i}E_{i}$, and each $E_{i}$ is diffeomorphic to $\R^{3}$ cutting off a ball $B_{R}$, and on each $E_{i}$, $e$ is the pull back of the standard Euclidean metric on $\R^{3}$. Here $\Si_{int}$ is called the \emph{interior region}, $\Si_{ext}$ the \emph{exterior region}, and each $E_{i}$ an \emph{end}. Each end $E$ has a coordinate system $\{x_{i}:\ i=1, 2, 3\}$ inherited from $\R^{3}$. Let $r=\sqrt{\sum_{i}(x_{i})^{2}}$. $(\Si, g, k)$ is said to be \emph{asymptotically flat}, abbreviated as (\textbf{AF}), if $(\Si, e)$ is Enclidean at infinity for some $e$, and there exists an $\al>\frac{1}{2}$, such that under coordinates $\{x_{i}:\ i=1, 2, 3\}$,
\begin{equation}\label{decay of g and k}
g_{ij}=\de_{ij}+o_{2}(r^{-\al}),\ k_{ij}=o_{1}(r^{-1-\al}).
\end{equation}
Under these conditions, the ADM mass is defined as,
\begin{equation}\label{ADM mass}
m=\lim_{r\rightarrow\infty}\frac{1}{16\pi}\int_{S_{r}}(g_{ij, i}-g_{ii, j})\nu^{j}d\si(r),\ 
\end{equation}
where $g_{ij, k}=\frac{\partial g_{ij}}{\partial x_{k}}$ and $\nu^{j}$ is the Euclidean unit outer normal of $S_{r}$ with $d\si(r)$ the surface element of $S_{r}$. The famous positive mass theorem by Schoen and Yau \cite{SY1}\cite{SY2} and Witten \cite{Wi} says that $m\geq0$ under the dominant energy condition.

If the initial data set $(\Si, g, k)$ is axisymmetric(\cite{D1}\cite{Chrus1}) under a Killing vector field $\xi$, i.e.
\begin{equation}\label{symmetry}
\mL_{\xi}g=0,\ \mL_{\xi}k=0,
\end{equation}
where $\mL$ denotes the Lie derivative, we also have a well-defined angular momentum $J$(\cite{D1}\cite{W}) of a close 2-surface $S\subset\Si$
\begin{equation}\label{angular momentum}
J(S)=\frac{1}{8\pi}\int_{S}\pi_{ij}\xi^{i}\nu^{j}d\si_{g},
\end{equation}
where $\pi_{ij}=k_{ij}-tr_{g}(k)g_{ij}$ is divergence free by (\ref{VCE}), and $\nu$, $d\si_{g}$ are, respectively the unit outer normal of $S$ and surface element w.r.t. $g$.

\subsection{Ideas and main results}\label{ideas and main results}

In this paper, we will prove the mass angular momentum inequality for certain axisymmetric, AF vacuum data $(\Sigma, g, k)$ with small $tr_{g}k$, especially we partially solved the question asked by Schoen. We will use the full Einstein equations and a perturbation method. Given an AF vacuum data $(\Si, g, k)$, we will solve the boost problem of (EVE) for $(\Si, g, k)$ as \cite{ChOM}\cite{BaChrusOM} to get a spacetime $(\V, \gamma)$, where $\V$ is a subset of $\Sigma\times\R$ which grows linearly at infinity. Given a function $u$ defined on $\Sigma$, the graph $Graph_{u}=\{(x, u(x))\in\Si\times\R, x\in\Sigma\}$ of $u$ lies inside $\V$, when $|u|$ has roughly sub-linear growth. We want to find a solution to $H_{u}=0$, where $H_{u}$ is the mean curvature of $Graph_{u}$ w.r.t. $(\V, \ga)$. Now fix a 3-manifold $(\Si, e)$ Euclidean at infinity, we can construct a mapping $\mH$ which takes the triple $(g, k, u)$ to the mean curvature $H_{u}$, i.e. $\mH: (g, k, u)\rightarrow H_{u}$. Viewing $(g, k)$ as parameters and $u$ as unknown function, our equation changes to
\begin{equation}\label{maximal surface equation}
\mH(g, k, u)=0.
\end{equation}
When $(g, k)$ is maximal, $u\equiv 0$ is a solution to (\ref{maximal surface equation}). So we can use the inverse function theorem to solve $\mH(g, k, u)=0$ when $tr_{g}k$ is small enough. From now on, we always assume $s\in\N$ and $\de\in\R$. Using notations from Section \ref{preliminaries}, we have

\begin{definition}\label{definition of VC and MVC}
Fix a 3-dimensional manifold $(\Si, e)$ which is Euclidean at infinity.
\vspace{-6pt}
\begin{itemize}
\addtolength{\itemsep}{-0.7em}
\item[(1)] The \emph{vacuum constraint data sets} $\VC_{s+1, \de+\frac{1}{2}}(\Si)$ is defined to be the set of solutions $(g, k)$ to (\ref{VCE}), such that $(g-e, k)\in H_{s+1, \de+\frac{1}{2}}(\Si)\times H_{s, \de+\frac{3}{2}}(\Si)$.

\item[(2)] The \emph{maximal vacuum constraint data sets} $\MVC_{s+1, \de+\frac{1}{2}}(\Si)$ is defined to be the subset of $\VC_{s+1, \de}(\Si)$ satisfying $tr_{g}k=0$.
\end{itemize}
\vspace{-6pt}
Inside $\VC_{s+1, \de+\frac{1}{2}}(\Si)$ and $\MVC_{s+1, \de+\frac{1}{2}}(\Si)$, we use the topology induced by the Sobolev norms of $H_{s+1, \de+\frac{1}{2}}(\Si)\times H_{s, \de+\frac{3}{2}}(\Si)$ as in Definition \ref{def for weighted sobolev space2}.
\end{definition}

\begin{definition}\label{definition of axisymmetric}
A simply connected 3-manifold $(\Si, e)$ which is Euclidean at infinity is called \emph{axisymmetric}, if
\vspace{-6pt}
\begin{itemize}
\setlength{\topsep}{0bp} 
\addtolength{\itemsep}{-0.7em}
\item[(1)] $\Si$ is diffeomorphic to $\R^{3}$ minus some points $\{a_{k}\}_{k=1}^{l-1}$ on the $z-$axis $\Ga=\{(\rho, \varphi, z)\in\R^{3}:\ \rho=0\}$, with one end modeled by a neighborhood of $\infty$, and other ends by a neighborhood of $a_{k}$ with coordinates given by a Kelvin transformation: $\{x^{\pr}=\frac{x-a_{k}}{|x-a_{k}|^{2}}\}$;
\item[(2)] $\mL_{\partial_{\varphi}}e=0$, where $\varphi$ is the azimuth of the cylindrical coordinate $\{\rho, \varphi, z\}$.
\end{itemize}
\end{definition}

\begin{remark}
Near $\infty$, $e$ is given by the Euclidean metric $ds_{0}^{2}$, and near each pucture $a_{k}$, $e$ is the pull back of the Euclidean metric by the Kelvin transformation, i.e. $e=\frac{1}{|x|^{4}}ds_{0}^{2}$. In fact, by Chru\'sciel's reduction in \cite{Chrus1}, any simply connected, axisymmetric, AF vacuum data $(\Si, g)$ has the underlying topology $\Si$ given by $\R^{3}$ minus finitely many points on the $z$ axis, with the Killing vector field $\frac{\partial}{\partial\varphi}$.
\end{remark}

\begin{definition}\label{definition of axisymmetric2}
Given $(\Si, e)$ as in Definition \ref{definition of axisymmetric}.
\vspace{-6pt}
\begin{itemize}
\addtolength{\itemsep}{-0.7em}
\item[(1)] An initial data set $(g, k)$ is called \emph{axisymmetric}, if symmetry conditions (\ref{symmetry}) hold for the Killing vector field $\xi=\frac{\partial}{\partial\varphi}$.
\item[(2)] $\VC^{a}_{s+1, \de+\frac{1}{2}}(\Si)$ and $\MVC^{a}_{s+1, \de+\frac{1}{2}}(\Si)$ are the axisymmetric subset of $\VC_{s+1, \de+\frac{1}{2}}(\Si)$ and $\MVC_{s+1, \de+\frac{1}{2}}(\Si)$ respectively.
\end{itemize}
\end{definition}

The following Theorem is one of our main results, which is a summarization of Theorem \ref{main theorem first part}, Lemma \ref{regularity of the maximal surface}, Lemma \ref{conservation of mass} and Theorem \ref{main theorem second part}.

\begin{theorem}\label{main theorem}\emph{(\textbf{Main Theorem 2})}
Given $s\geq 4$, $-2<\de<-1$.
\vspace{-5pt}
\begin{itemize}
\addtolength{\itemsep}{-0.7em}
\item[(i)] Let $(\Si, e)$ be a 3-dimensional manifold which is Euclidean at infinity. For any $(g, k)\in\VC_{s+2, \de+\frac{1}{2}}(\Si)$, where $\la e\leq g\leq\la^{-1}e$ for some $\la>0$, there exists a small number $\ep$ depending only on $\la$ and $\|g-e\|_{H_{s+2, \de+\frac{1}{2}}(\Si)}+\|k\|_{H_{s+1, \de+\frac{3}{2}}(\Si)}$, such that if $\|tr_{g}k\|_{H_{s-2, \de+\frac{3}{2}}(\Si)}\leq\ep$, then there exists a spacetime $(\V, \ga)$ solving the (EVE), and a function $u\in H_{s+2, \de-\frac{1}{2}}(\Si)$ solving the maximal surface equation (\ref{maximal surface equation}) inside $(\V, \ga)$. The induced metric $g_{u}$ and second fundamental form $k_{u}$ of $Graph_{u}$ satisfy $(g_{u}, k_{u})\in\MVC_{s+1, \de+\frac{1}{2}}(\Si)$.
\item[(ii)] If $-\frac{3}{2}<\de<-1$, the ADM mass of $(\Si, g_{u}, k_{u})$ is the same as that of $(\Si, g, k)$.
\item[(iii)] If $(\Si, e, g, k)$ is simply connected, axisymmetric, then $u$ can be chosen to be axisymmetric, hence $(\Si, g_{u}, k_{u})$ is axisymmetric, and has the same angular momentum as $(\Si, g, k)$.
\end{itemize}
\end{theorem}

\begin{remark}
The weight $\de$ corresponds the decay $g\sim e+o(r^{-(\de+2)})$ and $k\sim o(r^{-(\de+3)})$ by the Sobolev embedding lemma \ref{properties of weighted sobolev case2}. $(g_{u}, k_{u})$ is always assumed to be pulled back to $\Si$ by the graphical map $F_{u}: x\rightarrow(x, u(x))$.
\end{remark}
\begin{remark}
The order of regularity of our final solution $(g_{u}, k_{u})$ decreases by $1$ than our starting data $(g, k)$. This is due to the fact that the restriction of $H_{s}$ Soblev functions on a spacetime to a hypersurface decreases the regularity by $1$ (see Lemma \ref{restriction}).
\end{remark}

Our main Theorem \ref{mass angular momentum corollary} is then a corollary of the above theorem.

\vspace{6pt}
\noindent\textbf{Proof of Theorem \ref{mass angular momentum corollary}:}
Let $u$ be the solution given in part $(iii)$ of Theorem \ref{main theorem}. Then the induced maximal data $(g_{u}, k_{u})\in\MVC^{a}_{s+1, \de+\frac{1}{2}}(\Si)$, and the ADM mass $m$ and angular momentum $J$ of $(g, k)$ and $(g_{u}, k_{u})$ are the same. Now by Sobolev embedding lemma \ref{properties of weighted sobolev case2}, $(g_{u}-e, k_{u})\in C^{s-1}_{\be}(\Si)\times C^{s-2}_{\be+1}(\Si)$ for some $\frac{1}{2}<\be<\de+2<1$. So $(\Si, g_{u}, k_{u})$ is an axisymmetric, maximal vacuum data, with asymptotic conditions $g_{u}=\de+O_{s-1}(\frac{1}{r^{\be}})$ and $k_{u}=O_{s-2}(\frac{1}{r^{\be+1}})$, so the mass angular momentum inequality in \cite{SZ} holds on $(\Si, g_{u}, k_{u})$. Hence $m\geq\sqrt{|J|}$. \quad $\Box$ 
\vspace{6pt}

The paper is organized as follows: In Section \ref{preliminaries} we will review the weighted Sobolev space theories covered by \cite{ChOM}\cite{Ch1}\cite{ChoCH}\cite{Ba2} and the geometry of hypersurfaces in $3+1$ dimension Lorentzian spaces. In Section \ref{boost evolution} we will extend the boost theory in \cite{ChOM}\cite{Ch1} to the case of multi-ends. In Section \ref{perturbatoin method} we will set up a perturbation problem for the mean curvature of graphs. We will take initial data sets as parameters and use linear theory in \cite{Ba2}\cite{L}and\cite{ChoCH} and the Quantitative Inverse Function Theorem \ref{quantitative inverse function theorem}. Finally we will prove the main results in Section \ref{Proof of main theorem}.\\

\noindent\textbf{Acknowledgement:} The author would like to express his gratitude to his advisor Professor Richard Schoen for all of his helpful guidance and constant encouragement. He would like to thank Professor Rafe Mazzeo and Professor Leon Simon for lots of useful talks. He would also like to thank his friend Pin Yu for talking a lot about the hyperbolic equations.


\section{Preliminaries}\label{preliminaries}

In this section, we give some preliminary results on the weighted Sobolev space theories and the geometry of hypersurfaces in Lorentz spaces. 

\subsection{Weighted Sobolev space theories}\label{weighted sobolev space}

Here we give our definition of the weighted sobolev space. Most of the results here can be found in \cite{ChoCH}\cite{ChOM} and \cite{Ch1}. We will mainly talk about two types of domains.

\noindent\textbf{Type 1 domain: sub-domain of $\R^{3}$.}

Let $U$ be an open set in $\R^{n}$, $\si(x)=(1+|x|^{2})^{1/2}$ for $x\in\R^{n}$, and $V$ a finite dimensional vector space. Given $s\in\N$, $\de\in\R$.
\begin{definition}\label{def for weighted sobolev space}
$C^{s}_{\de}(U)$ is the Banach space of $C^{s}$ functions $u: U\rightarrow V$, with finite norm
$$\|u\|_{C^{s}_{\de}(U)}=\sup_{U}\big\{\sum_{|\al|\leq s}\si^{\de+|\al|}|D^{\al}u|\big\}.$$
$H_{s, \de}(U)$ is the class of functions $u: U\rightarrow V$, with weak derivatives up to order $s$, such that $\si^{\de+|\al|}D^{\al}u\in L^{2}(U)$ for all $\al\leq s$. $H_{s, \de}(U)$ is a Hilbert space with inner product:
$$\lan u_{1}, u_{2}\ran_{H_{s, \de}(U)}=\sum_{|\al|\leq s}\lan\si^{\de+|\al|}D^{\al}u_{1}, \si^{\de+|\al|}D^{\al}u_{2}\ran_{L^{2}(U)}.$$
Then the norm is: $\|u\|_{H_{s, \de}(U)}=\lan u, u\ran^{1/2}_{H_{s, \de}(U)}$.
\end{definition}

Now we will list some properties of $H_{s, \de}(U)$, which can be found in \cite{ChOM}\cite{ChoCH} and \cite{Ch1}. Given $0<\ep\leq 1$, and $\phi_{\ep}: \R^{3}\rightarrow\R^{3}$ defined by $\phi_{\ep}(x)=\frac{x}{(\si(x))^{1-\ep}}$. An open subset $U\subset\R^{3}$ is said to have the \emph{extended cone property} if $\phi_{\ep}(U)$ has the cone property\footnote{See the remark under Definition 2.3 of \cite{ChOM}} for each $0<\ep\leq 1$.
\begin{lemma}\label{properties of weighted sobolev space}
Given $U$ satisfying the extended cone property,
\vspace{-4pt}
\begin{itemize}
\addtolength{\itemsep}{-0.7em}
\item[$(i)$]\emph{(embedding)}. If $s^{\pr}<s-\frac{n}{2}$ and $\de^{\pr}<\de+\frac{n}{2}$, the inclusion $H_{s, \de}(U)\subset C^{s^{\pr}}_{\de^{\pr}}(U)$ is continuous;

\item[$(ii)$]\emph{(multiplication)}. If $s\leq s_{1}, s_{2}$, $s<s_{1}+s_{2}-\frac{n}{2}$ and $\de<\de_{1}+\de_{2}+\frac{n}{2}$, the multiplication $(f_{1}, f_{2})\rightarrow f_{1}f_{2}$ is continuous: $H_{s_{1}, \de_{1}}(U)\times H_{s_{2}, \de_{2}}(U)\rightarrow H_{s, \de}(U)$.
\end{itemize}
\vspace{-4pt}
Hence $H_{s, \de}(U)$ is a Banach algebra if $s>\frac{n}{2}$ and $\de>-\frac{n}{2}$.
\end{lemma}



\noindent\textbf{Type 2 domain: manifold which is Euclidean at infinity.}

Let $(\Si, e)$ be an $n$ dimensional manifold which is Euclidean at infinity. Let $x=\{x^{i}\}$ be the local coordinates, where $\{x^{i}\}$ is the pull back of the standard coordinates from $\R^{n}\setminus B_{R}$ when restricted to $E_{i}$, and $e=ds^{2}_{0}=\sum_{i=1}^{n}(dx^{i})^{2}$ on $E_{i}$. Fix a point $O\in\Si_{int}$, and define a function on $\Si$ by
$$\si_{e}(x)=(1+d_{e}^{2}(x, O))^{1/2}.$$
Clearly $\si_{e}(x)$ is equivalent to $\si(x)=(1+|x|^{2})^{1/2}$ on each end $E_{i}$.

When we use $\Si$ to model an initial data set, the spacetime should have topology as a sub-domain of $\Si\times\R$. Using coordinates $(x^{i}, t)$ on $\Si\times\R$, it has a natural reference metric
\begin{equation}\label{e tilde}
\ti{e}=dt^{2}+e.
\end{equation}
For $\theta\in(0, 1]$, the \emph{boost region} $\Om_{\theta}$ is defined as,
\begin{equation}\label{Omega2}
\Om_{\theta}=\{(x, t)\in\Si\times\R:\ |t|\leq\theta\si_{e}(x)\}.
\end{equation}
On $\Om_{\theta}$, the distance function $d_{\ti{e}}(\cdot, O)$ is equivalent to $d_{e}(\cdot, O)$, so we can use $\si_{e}$ to define the weighted Sobolev space on $\Om_{\theta}$. Given a smooth tensor bundle $E\rightarrow \Si$ or $E\rightarrow \Om_{\theta}$ and $s\in\N$, $\de\in\R$.
\begin{definition}\label{def for weighted sobolev space2}
$C^{s}_{\de}(\Si)$ or $C^{s}_{\de}(\Om_{\theta})$ is the Banach space of $C^{s}$ sections $u: \Si\rightarrow E$, or $u: \Om_{\theta}\rightarrow E$, with finite norm
$$\|u\|_{C^{s}_{\de}(\Si(or\ \Om_{\theta}))}=\sup_{\Si(or\ \Om_{\theta})}\big\{\sum_{|\al|\leq s}\si_{e}^{\de+|\al|}|D^{\al}u|_{e(or\ \ti{e})}\big\}.$$
$H_{s, \de}(\Si)$ or $H_{s, \de}(\Om_{\theta})$ is the class of sections $u: \Si\rightarrow E$, or $u: \Om_{\theta}\rightarrow E$ with weak derivatives up to order $s$, such that $\si_{e}^{\de+|\al|}D^{\al}u\in L^{2}(\Si, e)(or\ L^{2}(\Om_{\theta}, \ti{e}))$ for all $\al\leq s$. $H_{s, \de}(\Si)$  or $H_{s, \de}(\Om_{\theta})$ is a Hilbert space with inner product:
$$\lan u_{1}, u_{2}\ran_{H_{s, \de}(\Si)(or\ H_{s, \de}(\Om_{\theta}))}=\sum_{|\al|\leq s}\lan\si_{e}^{\de+|\al|}D^{\al}u_{1}, \si_{e}^{\de+|\al|}D^{\al}u_{2}\ran_{L^{2}(\Si, e)(or L^{2}(\Om_{\theta}, \ti{e}))}.$$
Then the norm is: $\|u\|_{H_{s, \de}(\Si)(or\ H_{s, \de}(\Om_{\theta}))}=\lan u, u\ran^{1/2}_{H_{s, \de}(\Si)(or\ H_{s, \de}(\Om_{\theta}))}$.
\end{definition}
\begin{remark}
In fact, the definitions are independent of the choice of $e$ on $\Si_{int}$.
\end{remark}

\begin{lemma}\label{properties of weighted sobolev case2}
\emph{(Lemma 2.4, 2.5 in \cite{ChoCH}, Appendex 1 in \cite{Cho})}
\vspace{-6pt}
\begin{itemize}
\addtolength{\itemsep}{-0.7em}
\item[$(i)$]\emph{(embedding)}. If $s^{\pr}<s-\frac{n}{2}$, $\de^{\pr}<\de+\frac{n}{2}$, the inclusion $H_{s, \de}(\Si)\subset C^{s^{\pr}}_{\de^{\pr}}(\Si)$ is continuous;

\item[$(ii)$]\emph{(multiplication)}. If $s\leq s_{1}, s_{2}$, $s<s_{1}+s_{2}-\frac{n}{2}$, $\de<\de_{1}+\de_{2}+\frac{n}{2}$, the multiplication $(f_{1}, f_{2})\rightarrow f_{1}f_{2}$ is a continuous map: $H_{s_{1}, \de_{1}}(\Si)\times H_{s_{2}, \de_{2}}(\Si)\rightarrow H_{s, \de}(\Si)$, hence $H_{s, \de}(\Si)$ is a Banach algebra if $s>\frac{n}{2}$, $\de>-\frac{n}{2}$. Furthermore,
\begin{equation}\label{multiplication inequality}
\|f_{1}f_{2}\|_{H_{s, \de}(\Si)}\leq C\|f_{1}\|_{H_{s_{1}, \de_{1}}(\Si)}\|f_{2}\|_{H_{s_{2}, \de_{2}}(\Si)},
\end{equation}
where $C$ is a constant depending only on $\{n, s_{1}, s_{2}, \de_{1}, \de_{2}\}$.
\end{itemize}
\end{lemma}

Divide $\Om_{\theta}$ as $\Om_{\theta}=(\Om_{\theta})_{int}\cup_{i=1}^{l}(\Om_{\theta})_{i}$, where $(\Om_{\theta})_{i}=\{(x, t)\in\Om_{\theta}:\ x\in E_{i}\}$, and $(\Om_{\theta})_{int}$ the compliment. Now $(\Om_{\theta})_{int}$ is a compact manifold, and $(\Om_{\theta})_{i}\subset\R^{n+1}$ satisfies the extended cone property in the above section, and hence Lemma \ref{properties of weighted sobolev space}. By working separately on $(\Om_{\theta})_{i}$ and $(\Om_{\theta})_{int}$ as in \cite{ChoCH} using Lemma \ref{properties of weighted sobolev space}, we have similar results,

\begin{lemma}\label{properties of weighted sobolev case3}
\begin{itemize}
\addtolength{\itemsep}{-0.7em}
\item[$(i)$]\emph{(embedding)}. If $s^{\pr}<s-\frac{n+1}{2}$, $\de^{\pr}<\de+\frac{n+1}{2}$, the inclusion is $H_{s+1, \de}(\Om_{\theta})\subset C^{s^{\pr}}_{\de^{\pr}}(\Om_{\theta})$ is continuous;

\item[$(ii)$]\emph{(multiplication)}. If $s\leq s_{1}, s_{2}$, $s<s_{1}+s_{2}-\frac{n+1}{2}$, $\de<\de_{1}+\de_{2}+\frac{n+1}{2}$, then the multiplication $(f_{1}, f_{2})\rightarrow f_{1}f_{2}$ is a continuous map: $H_{s_{1}, \de_{1}}(\Om_{\theta})\times H_{s_{2}, \de_{2}}(\Om_{\theta})\rightarrow H_{s, \de}(\Om_{\theta})$, hence $H_{s, \de}(\Om_{\theta})$ is a Banach algebra if $s>\frac{n+1}{2}$, $\de>-\frac{n+1}{2}$.
\end{itemize}
\end{lemma}
Using ideas similar to the proof of Theorem 2.3 in \cite{ChOM} and Lemma \ref{properties of weighted sobolev case3}, we have
\begin{lemma}\label{composition2}
\emph{(composition)}. Given $\Om_{\theta}$, $\Om_{\theta^{\pr}}$ as above and $f: \Om_{\theta}\rightarrow\Om_{\theta^{\pr}}$ a differentiable map, such that $|Df|_{\ti{e}}\geq c>0$ and $f-id\in H_{s+1, \de-1}(\Om_{\theta})$ with $s>\frac{n+1}{2}$ and $\de>-\frac{n+1}{2}$, then for any $s^{\pr}\leq s+1$, $\de^{\pr}\in\R$, the composition $u\rightarrow u\circ f$ is an isomorphism as a map:
$$H_{s^{\pr}, \de^{\pr}}(f(\Om_{\theta}))\rightarrow H_{s^{\pr}, \de^{\pr}}(\Om_{\theta}).$$
\end{lemma}

Define the function $\tau(x, t)=\frac{t}{\si_{e}(x)}$. Denote the level surface of $\tau$ by $\Si_{\tau}=\{(x, t)\in\Si\times\R:\ \tau(x, t)=\tau\}$. Then $\Om_{\theta}$ has a foliation $\Om_{\theta}=\cup_{\tau\in(-\theta, \theta)}\Si_{\tau}$. The restriction norm is defined as:
\begin{equation}\label{restriction norm2}
\|u\|_{H_{s, \de}(\Si_{\tau}, \Om_{\theta})}=\big(\sum_{k=0}^{s}\|D^{k}_{t}u|_{\Si_{\tau}}\|_{H_{s-k, \de+k}(\Si)}^{2}\big)^{1/2}.
\end{equation}
Using ideas similar to the proof of Lemma 3.1 in \cite{Ch1}, we can get,
\begin{lemma}\label{restriction}
\emph{(restriction)}. $\forall \tau\in(-\theta, \theta)$, we have the following continuous inclusion:
$$H_{s+1, \de}(\Om_{\theta})\subset H_{s, \de+\frac{1}{2}}(\Si_{\tau}, \Om_{\theta}),$$
for every $s\in\N$ and $\de\in\R$.
\end{lemma}


\subsection{Geometry of hypersurface in Lorentzian space}

In this section, we will review the geometry of hypersurfaces in a Lorentzian space. We will mainly focus on the mean curvature of the hypersurface. Notation and part of the results here trace back to \cite{Ba1}, and all concepts of Lorentzian space can be found in \cite{W}. Let $(\V, \gamma)$ be a (3+1) dimensional Lorentzian space, with $\lan\cdot,\ \cdot\ran$ the metric pairing and $\nabla$ the connection. A smooth function $t\in C^{\infty}(\V)$ is called a \emph{time function} if $\nabla t$ is nonzero, and everywhere \emph{timelike}, i.e. $\lan\nabla t, \nabla t\ran <0$. We call a hypersurface $\Sigma$ \emph{spacelike} if the restriction of $\gamma$ to $\Sigma$ is Riemannian. In a local coordinate system $\{x^{i}, t\}$, where $t$ is a time function, the metric can be written as(See equation (2.12) of \cite{Ba1}):
\begin{equation}\label{metric form for gamma}
\gamma=-(\al^{2}-\be^{2})dt^{2}+2\be_{i}dx^{i}dt+g_{ij}dx^{i}dx^{j},
\end{equation}
where $\al$ is the lapse function, i.e. $\al^{2}=-\lan\nabla t, \nabla t\ran$, $g_{ij}$ a Riemmanian metric, and $\be=g^{ij}\be_{i}\partial_{j}$ the shift vector\footnote{See Chap 10 of \cite{W} for details.}. Here we use $\partial_{t}=\frac{\partial}{\partial t}$ and $\partial_{i}=\frac{\partial}{\partial x^{i}}$ as coordinate vectors. The inverse metric $\ga^{-1}$ is given by:
\begin{equation}\label{inverse for gamma}
\ga^{\mu\nu}=\begin{bmatrix}
-\frac{1}{\al^{2}} & \frac{\be^{i}}{\al^{2}}\\
\frac{\be^{j}}{\al^{2}} & g^{ij}-\frac{1}{\al^{2}}\be^{i}\be^{j}
\end{bmatrix},
\end{equation}
under coordinate system $\{t, x^{1}, x^{2}, x^{3}\}$.

We will denote the level surface of the time function $t$ by $\Sigma_{t}=\{p\in\V:\ t(p)=t\}$. Let $D$ be the gradient operator on $\Si_{t}$, and $div^{0}$ the divergence operator on $\Si_{t}$. The future-directed timelike unite normal $T$ of $\Si_{t}$ is given by\footnote{See Appendix \ref{appendix} for details.}:
\begin{equation}\label{unit normal for constant slice}
T=-\al\nabla t=\al^{-1}(\partial_{t}-\be),
\end{equation}
and the second fundamental form $k^{0}_{ij}$ and the mean curvature $H^{0}$ of the slice $\Si_{t}$ are given by,
\begin{equation}\label{2nd form of t-slice}
k^{0}_{ij}=\lan\partial_{i}, \nabla_{\partial j}T\ran=\frac{1}{2}\al^{-1}\partial_{t}g_{ij}-\frac{1}{2}\al^{-1}\mL_{\be}g_{ij},
\end{equation}
\begin{equation}
H^{0}=g^{ij}A^{0}_{ij}=\frac{1}{2}\al^{-1}g^{ij}\partial_{t}g_{ij}-\al^{-1}div^{0}(\be).
\end{equation}

Given a spacelike hypersurface $\Sigma$, we can always choose local coordinates $\{x^{i}, t\}$, such that $\Sigma$ is locally the $t=0$ level surface $\Sigma_{0}$. Given a smooth function $u\in C^{\infty}(\Sigma)$, we can study the graph of $u$, i.e. $Graph_{u}=\{(x^{i}, u(x))\}$ in local coordinates. So we call this $u$ the \emph{height function}. By extending $u$ parallel to $\V$ requiring that
\begin{equation}\label{extension of u}
\partial_{t}u=0,
\end{equation}
$Graph_{u}$ can be viewed as level surface of $(u-t)=0$. The unit normal of $Graph_{u}$ is\footnote{See Appendix \ref{appendix} for details.}:
\begin{equation}\label{mean curvature for level surface}
N=\nu(U+T),
\end{equation}
where
\begin{equation}\label{U}
U=\frac{\al Du}{1+\lan \be, Du\ran},\ \textrm{and}\ \ \nu=\frac{1}{(1-|U|_{g}^{2})^{1/2}}.
\end{equation}
So $Graph_{u}$ is spacelike iff $1-|U|_{g}^{2}>0$, i.e. $\nu$ well-defined. Define the canonical graphical diffeomorphism $F: \Sigma\rightarrow Graph_{u}$ by $F(x)=(x, u(t))$. Then $Graph_{u}$ has a local coordinate system $\{x^{i}:\ i=1, 2, 3\}$. The coordinate vector frame $\{\partial_{i}\}$ on $\Sigma$ is passed by $F$ to a local frame
\begin{equation}\label{local frame}
\al_{i}=\partial_{i}+u_{i}\partial_{t}:\ i=1, 2, 3,
\end{equation}
on $Graph_{u}$. Now denote $M=Graph_{u}$. Using this local coordinates, the restriction $\gamma|_{M}$ of $\gamma$ to $Graph_{u}$, denoting by $g_{M}=(g_{M})_{ij}dx^{i}dx^{j}$, is given by
\begin{equation}\label{graph metric}
(g_{M})_{ij}=g_{ij}+\be_{i}u_{j}+u_{i}\be_{j}-(\al^{2}-\be^{2})u_{i}u_{j}.
\end{equation}
Then the inverse metric matrix is calculated in the Appendix \ref{appendix} by equation (\ref{inverse graph metric calculation}) as:
\begin{equation}\label{inverse graph metric}
\begin{split}
(g_{M})^{ij} &=g^{ij}-\frac{1}{\al^{2}}\be^{i}\be^{j}+\frac{\nu^{2}}{\al^{2}}(\be-\al U)^{i}(\be-\al U)^{j}\\
             &=\ga^{ij}+\frac{\nu^{2}}{\al^{2}}(\be-\al U)^{i}(\be-\al U)^{j}.
\end{split}
\end{equation}
So the mean curvature $H_{u}$ of the graph $M$ is given by
\begin{equation}\label{Hu1}
H_{u}=(g_{M})^{ij}\lan\nabla_{\al_{i}}N, \al_{j}\ran_{\ga}.
\end{equation}


\section{Boost evolution}\label{boost evolution}

Fix a 3-manifold $(\Si, e)$, which is Euclidean at infinity. Let $\ti{e}=dt^{2}+e$ be the reference metric(\ref{e tilde}) on $\Si\times\R$. Given integer $s\geq 4$, and real number $\de>-2$, we consider vacuum constraint initial data sets $(\Si, g, k)$, such that $(g, k)\in\V_{s, \de+\frac{1}{2}}(\Si)$. Here boost evolution means that in the spacetime $(\V, \ga)$ which is evolved by (EVE) taking $(\Si, g, k)$ as initial data set, where $\V\subset\Si\times\R$, both the future and past temporal distance $\chi_{\pm}(x)$\footnote{See \cite{ChOM} for reference.} to the boundary of $\V$ is proportional to the space distance $\si(x)$ for $x\in\Si$, i.e. $\chi_{\pm}(x)\geq c\si(x)$ for $c>0$. We will extend the boost evolution on $\R^{3}$ in \cite{ChOM} to the case of $\Si$.

\subsection{Reduced Einstein equation and results on compact domain}

Let us review the reduction using harmonic gauge initially introduced by Y. Choquet-bruhat(see \cite{Cho}). Using $\{x^{i}: i=1, 2, 3\}$ as local coordinates on $\Si$, and $x^{\mu}=(x^{0}, x^{i})$, with $x^{0}=t$ as coordinates on $\V\subset\Si\times\R$, the Ricci curvature can be expressed as\footnote{See Sec. 4 of \cite{ChOM} and Chap. 10.2 of \cite{W}.}:
$$Ric^{\mu\nu}=R^{\mu\nu}_{h}+\frac{1}{2}(\ga^{\mu\al}D_{\al}\Ga^{\nu}+\ga^{\nu\al}D_{\al}\Ga^{\mu}),$$
where $\Ga^{\mu}_{\al\be}$ is the Christoffel symbol of $\ga$, $\Ga^{\mu}=\ga^{\al\be}\Ga^{\mu}_{\al\be}$, and
$$R^{\mu\nu}_{h}=\frac{1}{2}\{\ga^{\al\be}D_{\al}D_{\be}\ga^{\mu\nu}-B^{\mu\nu}(\ga, D\ga)\},$$
with $B^{\mu\nu}=P^{\mu\nu, \rho\si}_{\al\be, \ka\la}D_{\rho}\ga^{\al\be}D_{\si}\ga^{\ka\la}$, and $P$ is a rational function of $\ga^{\al\be}$. In fact, The Einstein vacuum equation $Ric_{\gamma}=0$ is a degenerated differential equation system due to its invariance under diffeomorphic transformations. \emph{Harmonic gauge} is used to fix this gauge freedom by Y. Choquet-bruhat, which means that we can choose $id: (\V, \ga)\rightarrow (\V, \ti{e})$ to be a wave map, i.e. $\Box_{(\ga, e)}id=0$\footnote{See Section 7.4 of Chap 6 in \cite{Cho}}. Denote
\begin{equation}
f^{\mu}=\Ga^{\mu}-\ga^{\al\be}\ti{\Ga}^{\mu}_{\al\be},
\end{equation}
to be the \emph{harmonic gauge vector}, where $\ti{\Ga}^{\mu}_{\al\be}$ the Christoffel symbol of $\ti{e}$. $f^{\mu}$ is the difference of two connections, hence a tensor, then harmonic gauge condition reduces to $f^{\mu}=0$, or:
\begin{equation}\label{harmonic coordinates}
\Box_{\gamma}x^{\mu}=-\ga^{\al\be}\ti{\Ga}^{\mu}_{\al\be},
\end{equation}
where $\Box_{\ga}$ is the Laplacian operator of the Lorentzian metric $\ga$, and $\Box_{\ga}x^{\mu}=-\Ga^{\mu}$.
Now under harmonic gauge (\ref{harmonic coordinates}), the (EVE)(\ref{EVE}) reduced to\footnote{See page 163 in \cite{Cho}.}
\begin{equation}\label{reduced EVE}
\ga^{\al\be}D_{\al}D_{\be}\ga^{\mu\nu}=B^{\mu\nu}(\ga, D\ga)+\frac{1}{2}\ga^{\al\be}\{\ga^{\mu\rho}\ti{R}^{\nu}_{\be\al\rho}+\ga^{\nu\rho}\ti{R}^{\mu}_{\be\al\rho}\},
\end{equation}
where $\ti{R}$ is the curvature of $\ti{e}$. The Cauchy data for these equations consist of:
\begin{equation}
\ga|_{\Si}=\phi,\ D_{t}\ga|_{\Si}=\psi.
\end{equation}
For given initial data set $(g, k)$, we need to construct Cauchy data $(\phi, \psi)$ by requiring $f^{\mu}|_{\Si}=(\Ga^{\mu}-\ga^{\al\be}\ti{\Ga}^{\mu}_{\al\be})|_{\Si}=0$. To fix the freedom in choosing a harmonic gauge, we require the coordinate system of $\V$ is Gaussian on $\Si$, which means:
\begin{equation}\label{initial data1}
\phi^{00}=-1,\ \phi^{0i}=0,\ \phi^{ij}=g^{ij}.
\end{equation}
The condition $(\Ga^{\mu}-\ga^{\al\be}\ti{\Ga}^{\mu}_{\al\be})|_{\Si}=0$ requires\footnote{See page 164 in \cite{Cho}.}:
\begin{equation}\label{initial data2}
\psi^{00}=-4tr_{g}k,\ \psi^{0i}=-(\Ga_{g}^{i}-g^{kj}\ti{\Ga}^{i}_{kj}),\ \psi^{ij}=2g^{ik}g^{jl}k_{kl}.
\end{equation}
Define a reference Lorentzian metric by
\begin{equation}
\ti{\eta}=-dt^{2}+e.
\end{equation}
When the initial data $(g-e, k)\in H_{s, \de+\frac{1}{2}}(\Si)\times H_{s-1, \de+\frac{3}{2}}(\Si)$, the Cauchy data (\ref{initial data1})(\ref{initial data2}) satisfy $(\phi-\ti{\eta}, \psi)\in H_{s, \de+\frac{1}{2}}(\Si)\times H_{s-1, \de+\frac{3}{2}}(\Si)$. In fact, by multiplication lemma \ref{properties of weighted sobolev case2}, $(g-e, k)\rightarrow(\phi-\ti{\eta}, \psi)$ is a continuous map $H_{s, \de+\frac{1}{2}}(\Si)\times H_{s-1, \de+\frac{3}{2}}(\Si)\rightarrow H_{s, \de+\frac{1}{2}}(\Si)\times H_{s-1, \de+\frac{3}{2}}(\Si)$.

To solve (EVE)(\ref{EVE}), we can first solve the reduced equation (\ref{reduced EVE}) by quasilinear theory(see Appendix 3 in \cite{Cho}, and Section 5 in \cite{ChOM}), and then show that the harmonic gauge is preserved. In fact, Bianchi identity and the reduced equation (\ref{reduced EVE}) imply that the harmonic gauge vector $f^{\mu}$ satisfies a linear equation\footnote{See page 167 in \cite{Cho} and Section 4 in \cite{ChOM}}:
\begin{equation}\label{harmonic gauge evolution}
\Box_{\ga}f^{\mu}+A(\ga, D\ga)Df=0.
\end{equation}
So we can use uniqueness of linear equations to show that $f^{\mu}\equiv 0$ since we chose $f^{\mu}|_{\Si}=0$, and the constraint equations (\ref{VCE}) implies that $\partial_{t}f^{\mu}|_{\Si}=0$\footnote{See page 167 in \cite{Cho}.}.\\

Now we summarize a local version of the existence and causal uniqueness theorem based on the interior region $\Si_{int}$ of $(\Si, e)$, which has dimension $n=3$. We can extend the interior region $\Si_{int}$ to contain the annuli $B_{2R}\setminus B_{R}$ of each end $E_{i}$ of $(\Si, e)$. Now define a causal set $(V_{int})_{\theta, \la}$ based on $\Si_{int}$ as follows:
\begin{equation}\label{Vint theta}
(V_{int})_{\theta, \la}=\{(x, t)\in \Si_{int}\times[-\theta, \theta]:\ |x|\leq 2R-\la|t|,\textrm{if $x\in E_{i}$}\},
\end{equation}
where $\theta\in(0, 1]$ and $\la\geq 2$ is a positive number. Now $(V_{int})_{\theta, \la}$ has a lateral boundary $L_{\theta, \la}=\{(x, t)\in(V_{int})_{\theta, \la}:\ |x|=2R-\la|t|\}$. When $\la$ is large enough depending only on $e$, $L^{+}_{\theta, \la}=L_{\theta, \la}\cap\{t\geq 0\}$(or $L^{-}_{\theta, \la}=L_{\theta, \la}\cap\{t\leq 0\}$) is spacelike and ingoing(or outgoing) w.r.t. $\ti{\eta}$, hence $((V_{int})_{\theta, \la}, \ti{\eta})$ is causal\footnote{See Definition 2.11 of Appendix 3 in \cite{Cho}.}.

Combining Theorem 7.4, Theorem 8.3 of Chap 6, and Corollary 4.8, Theorem 4.11, Theorem 4.13 of Appendix 3 in \cite{Cho}, and using a cutoff argument as in Theorem \ref{boost theorem}, we have the following well-known local existence and uniqueness theorem,
\begin{theorem}\label{evolution on compact set}
Given an integer $s\geq 4$. For a vacuum constraint data set $(\Si_{int}, g, k)$, with $(g-e, k)\in H_{s}(\Si_{int})\times H_{s-1}(\Si_{int})$, and $g\geq \la_{0}e$ for some $\la_{0}>0$, there exists $\theta>0$, $\la\geq 2$ and $C_{0}>0$ depending only on $\la_{0}$ and $\|g-e\|_{H_{s}(\Si_{int})}+\|k\|_{H_{s-1}(\Si_{int})}$, and a unique regularly sliced\footnote{See page 397 and page 585 in \cite{Cho}} Lorentzian metric $\ga$ solving the reduced EVE(\ref{reduced EVE}) on $(V_{int})_{\theta, \la}$, taking (\ref{initial data1})(\ref{initial data2}) as initial value which is given by $(g, k)$, such that $(\ga-\ti{\eta})\in H_{s}\big((V_{int})_{\theta, \la}\big)$, with $\|\ga-\ti{\eta}\|_{H_{s}((V_{int})_{\theta, \la})}\leq C_{0}$, and $L^{+}_{\theta, \la}$(or $L^{-}_{\theta, \la}$) is spacelike and ingoing(or outgoing) w.r.t $\ga$. Furthermore, $\ga$ is a solution of (EVE)(\ref{EVE}) under harmonic gauge.
\end{theorem}


\subsection{Boost evolution on manifold Euclidean at infinity}

We first modify the linear boost theory in \cite{ChOM} to the case based on an Euclidean end $E\cong\R^{n-1}\setminus B_{R}$. Let us fix a special type of boost regions. Denote $\bar{x}=(x^{1}, \cdots, x^{n-1})\in\R^{n-1}$, such that $x=(\bar{x}, t)\in\R^{n}$. Later on, we will denote the index for $t$-coordinates as $0$, while index for $\bar{x}$ as $i$ with $i=1, \cdots, n-1$. Let $\bar{\si}(\bar{x})=(1+|\bar{x}|^{2})^{1/2}$. For $\theta\in(0, 1/2]$, $\la\geq 2$ and a given end $E\cong \R^{n-1}\setminus B_{R}$, the boost region $V_{\theta, \la}$ based on $E$ is defined as:
\begin{equation}\label{Vtheta}
V_{\theta, \la}=\big\{(\bar{x}, t)\in\R^{n},\ \frac{|t|}{\bar{\si}(\bar{x})}<\theta,\  |\bar{x}|\geq R+\la|t|\big\},
\end{equation}
Define function $\tau$ as $\tau(x)=\frac{t}{\bar{\si}(\bar{x})}$. Then the level surface of $\tau$ is $E_{\tau}=\{x\in V_{\theta, \la}:\ \tau(x)\equiv\tau\}$. $V_{\theta, \la}$ has a foliation:
$$V_{\theta, \la}=\cup_{\tau\in(-\theta, \theta)}E_{\tau}.$$
The lateral boundary of $V_{\theta, \la}$ is defined as,
\begin{equation}\label{lateral boundary}
L_{\theta, \la}=\{(\bar{x}, t)\in V_{\theta, \la}:\ |\bar{x}|=R+\la|t|\}.
\end{equation}
Denote the upper part of $V_{\theta, \la}$ by $V_{\theta, \la}^{+}=\{(\bar{x}, t)\in V_{\theta, \la}:\ t\geq0\}$, then the boundary $\partial V^{+}_{\theta, \la}$ is constituted by $E$, $E_{\theta}$ and the upper lateral boundary $L^{+}_{\theta, \la}=L_{\theta, \la}\cap V^{+}_{\theta, \la}$. Similarly, we have $V^{-}_{\theta, \la}=\{(\bar{x}, t)\in V_{\theta, \la}:\ t\leq 0\}$ and lower lateral boundary $L^{-}_{\theta, \la}=L_{\theta, \la}\cap V^{-}_{\theta, \la}$. Clearly $V^{\pm}_{\theta, \la}$ and the slices $E_{\tau}$ satisfy the extended cone property in $\R^{n}$ and $\R^{n-1}$ respectively as in Section \ref{weighted sobolev space}, and hence satisfy Lemma \ref{properties of weighted sobolev space}.

We introduce a class of hyperbolic metrics on $V_{\theta, \la}$ using the foliation $\{E_{\tau}\}_{\tau\in(-\theta, \theta)}$. The function $\tau$ is in fact a time function on $(V_{\theta, \la}, \eta)$, where $\eta=-dt^{2}+\sum_{i=1}^{n-1}(dx^{i})^{2}$ is the Minkowski metric. Let $\tilde{n}_{\mu}$ be the unit future co-normal of the foliation $\{E_{\tau}:\ \tau\in(-\theta, \theta)\}$, given by
\begin{equation}\label{future normal}
\tilde{n}=\tilde{N}D\tau=\frac{1}{\sqrt{1-\tau^{2}|\bar{x}|^{2}\si^{-2}(\bar{x})}}(dt-\frac{\tau}{\si(\bar{x})}x^{i}dx^{i}),
\end{equation}
where $\tilde{N}$ is the lapse function for the foliation $\{E_{\tau}\}$, defined by: $\tilde{N}^{-2}=-\lan D\tau, D\tau\ran_{\eta}=\frac{1-\tau^{2}|\bar{x}|^{2}\si^{-2}(\bar{x})}{\si^{2}(\bar{x})}$. $\tilde{n}$ can be viewed as a standard calibration for the foliation $V_{\theta, \la}=\cup E_{\tau}$, which is used to define the ``regularity" of hyperbolicity. Denoting $|\cdot|$ as the standard Euclidean norm for tensors on $V_{\theta}$, we have\footnote{See also Definition 4.1 in \cite{ChOM}.}:

\begin{definition}\label{regular hyperbolicity}
A $C^{0}$ covariant symmetric 2-tensor $\ga^{\mu\nu}$ on $V_{\theta, \la}$ is called \emph{regularly hyperbolic}, if there exist positive numbers $a$, $b$ and $C$ such that in $V_{\theta, \la}$:
\vspace{-6pt}
\begin{itemize}
\addtolength{\itemsep}{-0.7em}
\item[(1)] $-\ga^{\mu\nu}\tilde{n}_{\mu}\tilde{n}_{\nu}\geq a$;

\item[(2)] for all tangent covectors $\zeta_{\mu}$ of $E_{\tau}$, i.e. $\ga^{\mu\nu}\zeta_{\mu}\tilde{n}_{\nu}=0$, we have $\ga^{\mu\nu}\zeta_{\mu}\zeta_{\nu}\geq b|\zeta|^{2}$;

\item[(3)] $|\ga|\leq C$;

\item[(4)] The upper(or lower) lateral boundary $L^{+}_{\theta, \la}$(or $L^{-}_{\theta, \la}$) is spacelike and ingoing(or out going) w.r.t. $\ga$, i.e. every timelike curve entering $V^{+}_{\theta, \la}$(or every timeline curve exiting $V^{-}_{\theta, \la}$) is past directed.
\end{itemize}
\vspace{-6pt}
The \emph{coefficient of regular hyperbolicity} of $\ga$ is defined as,
\begin{equation}\label{coefficient of regular hyperbolicity}
h=\max\{\frac{1}{a}, \frac{1}{b}, C\}.
\end{equation}
\end{definition}

\begin{remark}\label{remark of space like and ingoing}
Condition (4) implies that this type of $V_{\theta, \la}$ is a \emph{causal subset} based on $E$ w.r.t $\ga$\footnote{ See Definition 2.11 of Appendix 3 in \cite{Cho}.}. Here we briefly talk about the criterion for Condition (4) to be true. We mainly discuss the case $L^{+}_{\theta, \la}$, and $L^{-}_{\theta, \la}$ is similar. The defining function of $L^{+}$ is given by $l(\bar{x}, t)=\la t+R-|\bar{x}|$, so the normal co-vector of $L^{+}$ is given by $dl=\la dt-d\bar{r}$, where $\bar{r}=|\bar{x}|$. Now $dl=\la(dt-\frac{\tau\bar{r}}{\si(\bar{x})}d\bar{r})+(\la\frac{\tau\bar{r}}{\si(\bar{x})}-1)d\bar{r}=\la\sqrt{1-\tau^{2}|\bar{x}|^{2}\si^{-2}(\bar{x})}\ti{n}+(\la\frac{\tau\bar{r}}{\si(\bar{x})}-1)d\bar{r}$. So using the regularly hyperbolicity, we have $\ga(dl, dl)\leq \la^{2}(1-\theta^{2})\ga(\ti{n}, \ti{n})+\la(\theta\la-1)C\leq -a\la^{2}(1-\theta^{2})+C\la(\theta\la-1)<0$, when $\la$ is chosen large enough depending only on $a$ and $C$, hence on $h$.
\end{remark}
\begin{remark}\label{remark of regular hyperbolicity}
The set of regularly hyperbolic metrics on $V_{\theta, \la}$ is open in the space $C^{0}(V_{\theta, \la})$ of bounded continuous covariant symmetric 2-tensors. In fact, $\eta$ is regular hyperbolic with $a=1$, $b=1-\theta^{2}$ and $C=\sqrt{n}$, and $L_{\theta, \la}$ is space-like and ingoing w.r.t $\eta$ when $\la\geq 2$. Since the space-like and ingoing condition for $L_{\theta, \la}$ is an open condition, there exist a small $\epsilon>0$, depending only on $\theta$, $\la$ and $n$, such that any $C^{0}$ covariant symmetric 2-tensor $\ga$, with $|\ga-\eta|\leq\epsilon$, is regularly hyperbolic in $V_{\theta, \la}$.
\end{remark}

Now consider a family of linear differential operators of second order in $V_{\theta, \la}$:
\begin{equation}\label{weakly coupled linear hyperbolic operator}
Lu=\Si^{2}_{k=0}a_{k}\cdot D^{k}u,
\end{equation}
where $u$ and $Lu$ are $\R^{N}$ valued functions on $V_{\theta, \la}$, and $a_{k}$ are matrix valued functions. The following hypotheses are required to the existence theory:
\vspace{-6pt}
\begin{itemize}
\addtolength{\itemsep}{-0.7em}
\item \emph{Hypothesis} $(1)$(weak coupling and hyperbolicity)\label{hypothesis1}. $a_{2}=\ga Id,$ i.e. $(a_{2})^{\mu\nu I}_{J}=\ga^{\mu\nu}\de^{I}_{J}$, $\mu, \nu=0, \cdots, n-1$, $I, J=1, \cdots, N$, where $\ga$ is a regularly hyperbolic metric on $V_{\theta, \la}$.

\item \emph{Hypothesis} $(2)$(regularity)\label{hypothesis2}. There exist integers $s_{k}$ and real numbers $\de_{k}$, such that: $s_{k}>\frac{n}{2}+k-1,\ \de_{k}>2-k-\frac{n}{2}:\ 0\leq k\leq 2$, and (1): $a_{k}\in H_{s_{k}, \de_{k}}(V_{\theta, \la})$ for $k=0, 1$; (2) $\ga-\eta\in H_{s_{2}, \de_{2}}(V_{\theta, \la})$.
\end{itemize}

\begin{remark}
Now denote
\begin{equation}\label{s prime}
s^{\pr}=min_{0\leq k\leq 2}\{s_{k}\}+1,
\end{equation}
\begin{equation}\label{coefficient m}
m=\|\ga-\eta\|_{H_{s_{2}, \de_{2}}(V_{\theta, \la})}+\Si_{k=0}^{1}\|a_{k}\|_{H_{s_{k}, \de_{k}}(V_{\theta, \la})},
\end{equation}
\begin{equation}\label{coefficient mu}
\mu=\|\ga-\eta\|_{H_{s_{2}-1, \de_{2}+1/2}(E, V_{\theta, \la})}+\sum_{k=0}^{1}\|a_{k}\|_{H_{s_{k}-1, \de_{k}+1/2}(E, V_{\theta, \la})}.
\end{equation}
By the restriction Lemma \ref{restriction}, $\mu\leq c m$. Using the multiplication Lemma \ref{properties of weighted sobolev space}, the regularity hypothesis (2) implies that
$$L: H_{s+1, \de}(V_{\theta, \la})\rightarrow H_{s-1, \de+2}(V_{\theta, \la}),$$
is a continuous map for $1\leq s\leq s^{\pr}$ and $\de\in\R$.
\end{remark}

Then we have the existence and uniqueness theorem for linear systems.
\begin{theorem}\label{estimates and existence for linear system}
Let $L$ be a differential operator defined by (\ref{weakly coupled linear hyperbolic operator}) in $V_{\theta, \la}$, satisfying Hypotheses (1) and (2). Let $\be\in H_{s-1, \de+2}(V_{\theta, \la})$, $\phi\in H_{s, \de+\frac{1}{2}}(E)$ and $\psi\in H_{s-1, \de+\frac{3}{2}}(E)$, with $2\leq s\leq s^{\pr}$, $\de\in\R$. Then the Cauchy problem:
\begin{equation}\label{linear systems}
Lu=\be,\ \ u|_{\Si}=\phi,\ D_{t}u|_{\Si}=\psi,
\end{equation}
has a unique solution $u\in H_{s, \de}(V_{\theta, \la})$, and satisfies the estimates:
\begin{equation}\label{estimates for linear equation}
\|u\|_{H_{s, \de}(V_{\theta, \la})}\leq c\theta^{\frac{1}{2}}\big\{\|\phi\|_{H_{s, \de+\frac{1}{2}}(E)}+\|\psi\|_{H_{s-1, \de+\frac{3}{2}}(E)}+\|\be\|_{H_{s-1, \de+2}(V_{\theta, \la})}\big\},
\end{equation}
where $c$ is a continuous increasing function of $(\theta, h, m)$, and $h$, $m$ are defined by equations(\ref{coefficient of regular hyperbolicity})(\ref{coefficient m}) respectively.
\end{theorem}
\begin{proof}
It follows from the energy estimates Theorem \ref{boost estimates for linear system} in Appendix \ref{Linear boost estimates}, and similar approximation argument as in the proof of Theorem 5.1 in \cite{Ch1} and Theorem 4.1 in \cite{ChOM}.
\end{proof}

Now we extend the existence theory for the boost problem in \cite{ChOM} to $\Si$. Let $\Om_{\theta}$ be the boost region based on $\Si$ as defined in (\ref{Omega2}). We will construct a solution to the reduced EVE(\ref{reduced EVE}) in $\Om_{\theta}$ under the harmonic gauge. We deal with the boost evolution separately on the interior region $\Si_{int}$ and on each end $E_{i}$. On compact set $\Si_{int}$, we can use Theorem \ref{evolution on compact set}. On each end $E$, we can complete the initial data $(g, k)|_{E}$ to $\R^{3}$ and apply the boost theory in \cite{ChOM} to get existence. Then we can cut off the solution in the causal set based on the end $E$ by our linear Theorem \ref{estimates and existence for linear system}. Causal uniqueness(see Corollary 4.8 of Appendix 3 in \cite{Cho}) tells us that the solutions we got based on $\Si_{int}$ and $E_{i}$'s match together to a global solution.

\begin{theorem}\label{boost theorem}
For $s\geq 4$, $\de>-2$. Given vacuum data $(g, k)\in\VC_{s, \de+\frac{1}{2}}(\Si)$, with $g\geq\la_{0} e$ for some $\la_{0}>0$, there exits a $\theta\in(0, 1)$ and a $C_{0}>0$ depending only on $\la_{0}$, $\|g-e\|_{H_{s, \de+\frac{1}{2}}(\Si)}+\|k\|_{H_{s-1, \de+\frac{3}{2}}(\Si)}$, and a unique Lorentzian metric $\ga$ solving the reduced EVE (\ref{reduced EVE}) on $\Om_{\theta}$, which has Cauchy data $(\phi, \psi)$ on  $\Si$ given by $(g, k)$ in (\ref{initial data1})(\ref{initial data2}), such that $(\ga-\ti{\eta})\in H_{s, \de}(\Om_{\theta})$, and $\|\ga-\ti{\eta}\|_{H_{s, \de}(\Om_{\theta})}\leq C_{0}$. Furthermore $\ga$ is the solution to EVE (\ref{EVE}) under harmonic gauge.
\end{theorem}
\begin{proof}
We first focus on a fixed end $E$. In fact, we can extend $(g, k)|_{E}$ to $(\bar{g}, \bar{k})$ on $\R^{3}$ by a cut and paste method, such that $(\bar{g}, \bar{k})=(g, k)$ on $E$ with $\bar{g}\geq\bar{\la}\de$, where $\bar{\la}\geq c^{-1}\la_{0}$ and $\|\bar{g}-\de\|_{H_{s, \de+\frac{1}{2}}(\R^{3})}+\|\bar{k}\|_{H_{s-1, \de+\frac{3}{2}}(\R^{3})}\leq c(\|g-e\|_{H_{s, \de+\frac{1}{2}}(E)}+\|k\|_{H_{s-1, \de+\frac{3}{2}}(E)})$ for some fixed $c>1$. By Lemma 5.1 and Theorem 6.1 in \cite{ChOM}, there exist $C_{1}>0$ and $\theta_{1}\in(0, 1)$ depending only on $\bar{\la}$ and $\|\bar{g}-\de\|_{H_{s, \de+\frac{1}{2}}(\R^{3})}+\|\bar{k}\|_{H_{s-1, \de+\frac{3}{2}}(\R^{3})}$, and a unique solution $\bar{\ga}$ to the reduce EVE (\ref{reduced EVE}) on $\Om_{\theta_{1}}$, taking on $\R^{3}$ the Cauchy data $(\bar{\phi}, \bar{\psi})$ given by $(\bar{g}, \bar{k})$ as in (\ref{initial data1})(\ref{initial data2}) where the Christoffel symbol for $\R^{3}$ is $\ti{\Ga}|_{\R^{3}}=0$, and $\|\bar{\ga}-\eta\|_{H_{s, \de}(\Om_{\theta_{1}})}<C_{1}$. Here $\Om_{\theta_{1}}$ is the boost region (\ref{Omega2}) when $\Si=R^{3}$. Furthermore, $\bar{\ga}$ is regularly hyperbolic\footnote{Here regularly hyperbolicity is given in Definition 4.1 in \cite{ChOM}, which only requires the first three conditions in Definition  \ref{regular hyperbolicity}.}, with the coefficient of regularly hyperbolicity $h_{1}$ depending only on  $\bar{\la}$ and $\|\bar{g}-\de\|_{H_{s, \de+\frac{1}{2}}(\R^{3})}+\|\bar{k}\|_{H_{s-1, \de+\frac{3}{2}}(\R^{3})}$.

We claim that there exists a $\la_{1}>2$ depending only on $h_{1}$, such that $\bar{\ga}$ is regularly hyperbolic on $V_{\theta_{1}, \la_{1}}$. The first three conditions in Definition \ref{regular hyperbolicity} are naturally satisfied since $\bar{\ga}$ is regularly hyperbolic in $\Om_{\theta_{1}}$(See Definition 4.1 in \cite{ChOM}). Condition (4) is true if we take take $\la_{1}$ large enough depending only on the regularly hyperbolicity $h_{1}$ of $\bar{\ga}$ as discussed in Remark \ref{remark of space like and ingoing}.

Then we claim that $\bar{\ga}$ is a solution of (EVE)(\ref{EVE}) in harmonic gauge inside the causal set $V_{\theta_{1}, \la_{1}}$. In fact, since $(g, k)$ is a solution of (VCE)(\ref{VCE}) on $E$, the harmonic gauge condition $f^{\mu}=\Ga^{\mu}_{\bar{\ga}}=0$ and $\partial_{t}f^{\mu}=0$ on $E$ are satisfied by the choice of initial conditions (\ref{initial data1})(\ref{initial data2}). Notice that $f$ satisfies a linear equation (\ref{harmonic gauge evolution}), which satisfies the requirement of Theorem \ref{estimates and existence for linear system} by argument on page 293 in \cite{ChOM}. Hence the harmonic gauge vector $f=0$ on $V_{\theta_{1}, \la_{1}}$ by the estimate (\ref{estimates for linear equation}) in Theorem \ref{estimates and existence for linear system}, hence $\bar{\ga}$ is a solution of EVE (\ref{EVE}) on $V_{\theta_{1}, \la_{1}}$.

Now denote the restriction $\bar{\ga}$ to $V_{\theta_{1}, \la_{1}}$ by $\ga$. We claim that $(V_{\theta_{1}, \la_{1}}, \ga)$ is uniquely determined by $(g, k)|_{E}$ when $\ga$ is regularly hyperbolic on $V_{\theta_{1}, \la_{1}}$.
Suppose $\ga_{1}$ and $\ga_{2}$ are two such solutions of reduced EVE (\ref{reduced EVE}) as above with initial value given by (\ref{initial data1})(\ref{initial data2}) from vacuum data $(g_{1}, k_{1})$ and $(g_{2}, k_{2})$ respectively. Then $\|\ga_{i}-\eta\|_{H_{s, \de}(V)}$ are uniformly bounded by the corresponding norm of $(g_{i}-\eta, k_{i})$. Now subtract the reduced EVE (\ref{reduced EVE}) satisfied by $\ga_{1}$ and $\ga_{2}$:
\begin{equation}\label{subtraction of ga1 and ga2}
\ga_{1}^{\al\be}D_{\al}D_{\be}(\ga_{1}^{\mu\nu}-\ga_{2}^{\mu\nu})-(D^{2}\ga_{2})(\ga_{2}-\ga_{1})-\big(B(\ga_{1}, D\ga_{1})-B(\ga_{2}, D\ga_{2})\big)=0,
\end{equation}
where (see equations (4.4)(4.5) in \cite{ChOM})
$$B(\ga_{1}, D\ga_{1})-B(\ga_{2}, D\ga_{2})=P(\ga_{1})(D\ga_{1})^{2}-P(\ga_{2})(D\ga_{2})^{2}$$
$$=(P(\ga_{1})-P(\ga_{2}))(D\ga_{1})^{2}+P(\ga_{2})(D\ga_{1}+D\ga_{2})(D\ga_{1}-D\ga_{2}).$$
Here $P$ is a rational function of $\ga$. Using the multiplication lemma \ref{properties of weighted sobolev space},  $(D\ga_{1})^{2},\ P(\ga_{2})(D\ga_{1}+D\ga_{2})\in H_{s-1, \de+1}(V)$. Using the mean value inequality, and the Soblev embedding lemma \ref{properties of weighted sobolev space}, we have the pointwise estimates:
$$|P(\ga_{1})-P(\ga_{2})|\leq C|\ga_{1}-\ga_{2}|,$$
where $C$ depends only on $\|\ga_{i}-\eta\|_{H_{s, \de}(V)}$, $i=1, 2$.
Now viewing equation (\ref{subtraction of ga1 and ga2}) as a differential equation for $(\ga_{1}-\ga_{2})$, and using the first energy estimate Lemma \ref{fundamental energy estimates2} in Appendix \ref{Linear boost estimates}, we have
$$\|\ga_{1}-\ga_{2}\|_{H_{1, \de+\frac{1}{2}}(E_{\tau}, V)}\leq C\|\ga_{1}-\ga_{2}\|_{H_{1, \de+\frac{1}{2}}(E, V)}\leq C(\|g_{1}-g_{2}\|_{H_{1, \de+\frac{1}{2}}(E)}+\|k_{1}-k_{2}\|_{H_{0, \de+\frac{3}{2}}(E)}).$$
Hence the uniqueness is true.

Combing all the above, we get a unique regularly hyperbolic solution $\ga$ to the (EVE) under harmonic gauge on $V_{\theta_{1}, \la_{1}}$, where $\theta_{1}$, $\la_{1}$ and $\|\ga-\eta\|_{H_{s, \de}(V)}$\footnote{The bound for $(\ga-\eta)$ also comes directly by Theorem \ref{estimates and existence for linear system}.} depend only on $\la_{0}$ and $\|g-e\|_{H_{s, \de+\frac{1}{2}}(E)}+\|k\|_{H_{s-1, \de+\frac{3}{2}}(E)}$.

Now extend $\Si_{int}$ to include annuli $B_{r}\setminus B_{R}\subset E_{i}$, and take the solution $\ga$ inside the causal set $(V_{int})_{\theta_{0}, \la_{0}}$ based on $\Si_{int}$ by Theorem \ref{evolution on compact set}. We can combine it with all the solutions $(V_{\theta_{i}, \la_{i}}, \ga)$ on each end $E_{i}$. Now causal uniqueness(See Corollary 4.8 of Appendix 3 in \cite{Cho}) implies that they coincide in the intersection of $(V_{int})_{\theta_{0}, \la_{0}}$ and $V_{\theta_{i}, \la_{i}}$, since $(V_{int})_{\theta_{0}, \la_{0}}\cap V_{\theta_{i}, \la_{i}}$ is a causal set based on $\Si_{int}\cap E$ w.r.t. $\ga$ by our construction. So by choosing the smallest $\theta$, such that $\Om_{\theta}\subset (V_{int})_{\theta_{0}, \la_{0}}\cup_{i=1}^{l}V_{\theta_{i}, \la_{i}}$, we get the conclusion.
\end{proof}


\section{Perturbation method}\label{perturbatoin method}

Here we will apply the Inverse Function Theorem(See \cite{Cha}\cite{R}) to get maximal graphs in the spacetime evolution of given AF vacuum data sets with small trace. Fix a 3-manifold $(\Si, e)$ which is Euclidean at infinity. We always assume $s\in\N$, $s\geq 4$, and $\de>-2$. Consider the vacuum data sets $(\Si, g, k)$, with $(g, k)\in\VC_{s+1, \de+\frac{1}{2}}(\Si)$. Let $(\V, \ga)$ be the boost evolution of $(g, k)$ given by Theorem \ref{boost theorem}, then we will study the graph of given function $u$ in the spacetime $(\V, \ga)$. We will take $(g, k)$ as parameters, and study the perturbation problem for the mean curvature $H_{u}$ of this graph. We will show that for appropriately chosen weighted Sobolev spaces, the linearization of $H_{u}$ with respect to $u$ is invertible in certain sense.

\subsection{Differentiability of mean curvature operator}

Given a vacuum data set $(g, k)\in\VC_{s+1, \de+\frac{1}{2}}(\Si)$, with $g\geq \la e$ for some $\la>0$. By Theorem \ref{boost theorem}, there exists a uniform $\theta\in(0, 1)$ and a uniform $C>0$, depending only on $\la$ and $\|g-e\|_{H_{s+1, \de+\frac{1}{2}}(\Si)}+\|k\|_{H_{s, \de+\frac{3}{2}}(\Si)}$, and a unique Lorentzian solution $\ga$ of the reduced EVE (\ref{reduced EVE}) on $\Om_{\theta}$, taking $(g, k)$ as initial data, and $\|\ga-\ti{\eta}\|_{H_{s+1, \de}(\Om_{\theta})}\leq C$. Moreover, from the proof of Theorem \ref{boost theorem}, the regularly hyperbolic coefficient $h$ of $\ga$ in each boost end $V_{\theta_{i}, \la_{i}}$, and the regularly sliced coefficient\footnote{See the constant $N$, $A$ and $B$ in Definition 11.8 in page 397 in \cite{Cho}} of $\ga$ in $(V_{int})_{\theta_{0}, \la_{0}}$ are all uniformly bounded by a constant depending only on $\la$ and the norm of $(g, k)$. Hence the determinant of $\ga^{\mu\nu}$ is bounded away from $0$ by a constant depending only on $\la$ and the norm of $(g, k)$.

Now let us summarize some properties of metric components of $\ga$.
\begin{lemma}\label{sobolev bound for metric coefficient}
For $s\geq 3$, $\de>-2$. Given a $(3+1)$ Lorentz metric $\ga^{\mu\nu}$ of form (\ref{inverse for gamma}) in $\Om_{\theta}$ with $(\ga-\ti{\eta})^{\mu\nu}\in H_{s, \de}(\Om_{\theta})$, if the determinant $det(\ga^{\mu\nu})\leq-\ti{\la}$ for some $\ti{\la}>0$, then $(\ga-\ti{\eta})_{\mu\nu}$ lies in $H_{s, \de}(\Om_{\theta})$, and in the metric form (\ref{metric form for gamma})(\ref{inverse for gamma}) of $\ga$, the components
$(\al^{-2}-1), (\al-1), \be^{i}, \be_{i}, g^{ij}-e^{ij}, g_{ij}-e_{ij}$ all lie in $H_{s, \de}(\Om_{\theta})$. Furthermore, their norms are all bounded by a constant depending only on $\ti{\la}$ and $\|\ga-\ti{\eta}\|_{H_{s, \de}(\Om_{\theta})}$.
\end{lemma}
\begin{proof}
The inverse matrix $\ga_{\mu\nu}=det(\ga^{\mu\nu})adj(\ga^{\mu\nu})$, where $adj(\ga^{\mu\nu})$ is the adjoint matrix of $\ga^{\mu\nu}$. Since $det(\ga^{\mu\nu})$ is bounded away from $0$ by $\ti{\la}$, the Banach algebra property(Lemma \ref{properties of weighted sobolev case3}) of $H_{s, \de}(\Om_{\theta})$ implies that $\ga_{\mu\nu}-\ti{\eta}_{\mu\nu}$ also lies in $H_{s, \de}(\Om_{\theta})$, with $\|\ga_{\mu\nu}-\ti{\eta}_{\mu\nu}\|_{H_{s, \de}(\Om_{\theta})}$ bounded by a constant depending only on $\ti{\la}$ and $\|(\ga-\ti{\eta})^{\mu\nu}\|_{H_{s, \de}(\Om_{\theta})}$. From the expression (\ref{metric form for gamma})(\ref{inverse for gamma}) of $\ga$ and the fact that $(\ga-\ti{\eta})^{\mu\nu},\ (\ga-\ti{\eta})_{\mu\nu}\in H_{s, \de}(\Om_{\theta})$, we know that $(\al^{2}-1),\ (\frac{1}{\al^{2}}-1),\ \be^{i},\ \frac{\be^{i}}{\al^{2}},\ (g_{ij}-e_{ij}),\ (g^{ij}-\frac{\be^{i}\be^{j}}{\al^{2}}-e^{ij})\in H_{s, \de}(\Om_{\theta})$ with their norms bounded by $\|(\ga-\ti{\eta})\|_{H_{s, \de}(\Om_{\theta})}$. So $\al^{2}$ is bounded both from below and above by certain constant. By Taylor's expansion $|\al-1|=|\sqrt{1+(\al^{2}-1)}-1)|\leq C|\al^{2}-1|$, hence is $L^{2}_{\de}$ integrable. For higher order derivatives of $(\al-1)$, we can use the multiplication Lemma \ref{properties of weighted sobolev case3} and the bound of $(\al^{2}-1)$ to show that $D^{\mu}(\al-1)$ lies in $L_{s-|\mu|, \de+|\mu|}(\Om_{\theta})$. So $(\al-1)$ lies in $H_{s, \de}(\Om_{\theta})$ and has the norm bounded by a constant depending only on $\ti{\la}$ and $\|(\ga-\ti{\eta})^{\mu\nu}\|_{H_{s, \de}(\Om_{\theta})}$.
\end{proof}

So the metric coefficients of out boost solution $\ga$ satisfy that $\{(\al-1), \be^{i}, \be_{i}, g^{ij}-e^{ij}, g_{ij}-e_{ij}\}\in H_{s+1, \de}(\Om_{\theta})$ with norms bounded by a constant depending only on the elliptic constant $\la$ of $g$ and $\|g-e\|_{H_{s+1, \de+\frac{1}{2}}(\Si)}+\|k\|_{H_{s, \de+\frac{3}{2}}(\Si)}$. By the Soblev embedding $H_{s+1, \de}(\Om_{\theta})\subset C^{2}_{\ka}(\Om_{\theta})$ for some $0<\ka<\de+2$, all the terms above are uniformly bounded.

Given $s_{1}\geq 3$ and $\de_{1}>-2$. Let $\B_{\rho}$ be a ball of radius $\rho$ containing scalar functions in $H_{s_{1}+1, \de_{1}-\frac{1}{2}}(\Si)$ with $\|u\|_{H_{s_{1}+1, \de_{1}-\frac{1}{2}}(\Si)}\leq\rho$. We can choose $\rho$ small enough, such that after embedding $\|u\|_{C^{2}_{\ka}(\Si)}\leq C\rho\leq\theta/2$ for some $-1<\ka<\de_{1}+1$, and,
\begin{equation}\label{condition (A)}
\textrm{Condition (A):}\qquad \|u(x)\|_{L^{\infty}}\leq (\theta/2)(\si(x))^{-\ka}<(\theta/2)\si(x).
\end{equation}
So $Graph_{u}=\{(x, u(x)): x\in\Si\}$ is a submanifold in $\Om_{\theta}$. Furthermore, $|Du|_{e}\leq C\rho(\si(x))^{-(\ka+1)}$. As $(\al-1), \be, (g-e)$ are all uniformly bounded, we can then choose $\rho$ small enough satisfying:
\begin{equation}\label{condition (B)}
\textrm{Condition (B):}\quad |Du|_{e}\leq\frac{1}{100},\ |\lan\be, Du\ran_{g}|\leq\frac{1}{2},\ |U|=|\frac{\al|Du|_{g}}{1+\lan\be, Du\ran_{g}}|\leq\frac{1}{2},
\end{equation}
where $U$ is defined in (\ref{U}). Then $Graph_{u}$ is spacelike and $\nu=\sqrt{1-|U|^{2}}$ is well-defined. So we can study the operator
\begin{equation}\label{mean curvature operator}
\mH: u\rightarrow H_{u},
\end{equation}
where $H_{u}$ is the mean curvature of $Graph_{u}$ given by (\ref{Hu1}).

Now we will show that composition is continuous as follows,
\begin{lemma}\label{composition at x u}
Given $s_{1}\geq 3$, $\de_{1}>-2$ and $\theta\in(0, 1)$. Consider $\B_{\rho}\subset H_{s_{1}+1, \de_{1}-\frac{1}{2}}(\Si)$ with $\rho$ small enough satisfying Conditions (A) as above for the $\theta$. Then the composition map:
$$(f, u)\rightarrow \ti{f}=f(x, u(x)+t),$$
is a continuous map $H_{s^{\pr}, \de^{\pr}}(\Om_{\theta})\times \B_{\rho}\rightarrow H_{s^{\pr}, \de^{\pr}}(\Om_{\theta/2})$, for $s^{\pr}\leq s_{1}+1$ and $\de^{\pr}\in\R$. Furthermore, when restricted to $Graph_{u}$,
$$(f, u)\rightarrow f(x, u(x))$$
is a continuous map $H_{s^{\pr}, \de^{\pr}}(\Om_{\theta})\times \B_{\rho}\rightarrow H_{s^{\pr}-1, \de^{\pr}+\frac{1}{2}}(\Si, \Om_{\theta/2})$.
\end{lemma}
\begin{proof}
Condition(A) (\ref{condition (A)}) implies that $|u(x)|\leq(\theta_{0}/2)\si(x)^{-\ka}$ for some $-1<\ka<\de_{1}+1$, so we can consider a well-defined map $F:\Om_{\theta}\rightarrow\Om_{\frac{3}{2}\theta}$, where $F: (x, t)\rightarrow (x, u(x)+t)$. Then $|DF|=1$, so $F$ is a diffeomorphism $\Om_{\theta}\rightarrow F(\Om_{\theta})$. Furthermore, $(F-id)(x, t)=(0, u(x))\in H_{s_{1}+1, \de_{1}}(\Om_{\theta})$. Now we can apply lemma \ref{composition2} to the mapping $F$, so $f\rightarrow\ti{f}=f\circ F$ is an isomorphism $H_{s^{\pr}, \de^{\pr}}(\Om_{\theta})\rightarrow H_{s^{\pr}, \de^{\pr}}(F(\Om_{\theta}))$. In fact, by the bound of $u$, we know that $F(\Om_{\theta})$ contains $\Om_{\theta/2}$, so clearly $\ti{f}$ lies in $H_{s^{\pr}, \de^{\pr}}(\Om_{\theta/2})$, and we have the continuity for the first factor $f$. For the second factor $u$, we only need to show that $u\rightarrow f(x, u(x)+t)$ is continuous $H_{s_{1}+1, \de_{1}-\frac{1}{2}}(\Si)\rightarrow L^{2}_{\de^{\pr}}(\Om_{\theta/2})$ for fixed $f\in L^{2}_{\de^{\pr}}(\Om_{\theta})$. Using multiplication lemma \ref{properties of weighted sobolev case3} recursively to higher derivatives as in the proof of Theorem 2.3 in \cite{ChOM} gives the continuity in $H_{s^{\pr}, \de^{\pr}}$. Suppose $u_{n}\rightarrow u$ in $H_{s_{1}+1, \de_{1}-\frac{1}{2}}$, hence $u_{n}\rightarrow u$ in $C^{0}_{\ka}$ for some $-1<\ka<\de_{1}+1$. To show the $L^{2}_{\de^{\pr}}$ continuity, we can approximate $f$ by compactly supported smooth function $g$ in $L^{2}_{\de^{\pr}}$, then $|f(x, u_{n}(x)+t)-f(x, u(x)+t)|\leq|f(x, u_{n}(x)+t)-g(x, u_{n}(x)+t)|+|g(x, u(x)+t)-f(x, u(x)+t)|+|g(x, u_{n}(x)+t)-g(x, u(x)+t)|$. The first and second terms can be chosen very small in $L^{2}_{\de^{\pr}}$, and the third one converge to $0$ in $L^{2}_{\de^{\pr}}$. So we get the continuity. For the restriction, we can directly apply the restriction lemma \ref{restriction} to $\ti{f}$.
\end{proof}

Moreover, we also have the differentiability w.r.t. $u$.
\begin{lemma}\label{differentiable}
Given $s_{1}\geq 3$, $\de_{1}>-2$, $\theta\in(0, 1)$, $\de^{\pr}\in\R$ and $f\in H_{s_{1}+1, \de^{\pr}}(\Om_{\theta})$. Consider $\B_{\rho}\subset H_{s_{1}+1, \de_{1}-\frac{1}{2}}(\Si)$ with $\rho$ chosen to satisfy Condition(A) in (\ref{condition (A)}) for the $\theta$. Then
\begin{equation}
\F: u\rightarrow f(x, u(x)),
\end{equation}
is continuous Fr\'echet differentiable as a map $\B_{\rho}\rightarrow H_{s_{1}-1, \de^{\pr}+\frac{1}{2}}(\Si)$. Furthermore, the Fr\'echet derivative is given by formal derivatives,
\begin{equation}
D_{u}\F(v)=\partial_{t}f(x, u(x))\cdot v,
\end{equation}
where $v\in H_{s_{1}+1, \de_{1}-\frac{1}{2}}(\Si)$.
\end{lemma}
\begin{proof}
Using lemma \ref{composition at x u}, we know that $f(x, u(x)+t)$ lies in $H_{s_{1}+1, \de^{\pr}}(\Om_{\theta/2})$, and $f(x, u(x))\in H_{s_{1}, \de^{\pr}+\frac{1}{2}}(\Si, \Om_{\theta/2})$. Hence $\partial_{t}f(x, t)\in H_{s_{1}, \de^{\pr}+1}(\Om_{\theta/2})$ and $\partial_{t}f(x, u(x))\in H_{s_{1}-1, \de^{\pr}+\frac{3}{2}}(\Si)$. To show that $\F$ is Fr\'echet differentiable(See Definition 1.1.1 in \cite{Cha}), we can first show Gateaux differentiable(See Definition 1.1.2 in \cite{Cha}), i.e.
\begin{equation}
\lim_{\tau\rightarrow 0}\frac{\|f(x, u(x)+\tau v(x))-f(x, u(x))-\partial_{t}f(x, u(x))(\tau v(x))\|_{H_{s_{1}-1, \de^{\pr}+\frac{1}{2}}(\Si)}}{\tau\|v(x)\|_{H_{s_{1}+1, \de_{1}-\frac{1}{2}}(\Si)}}=0,
\end{equation}
for any $v\in H_{s_{1}+1, \de_{1}-\frac{1}{2}}(\Si)$. Using Newton-Leibniz formula,
\begin{equation}
f(x, u(x)+\tau v(x))-f(x, u(x))=(\int_{s=0}^{1}\partial_{t}f(x, u(x)+s\tau v(x))ds)(\tau v(x)),
\end{equation}
Using the multiplication lemma (\ref{multiplication inequality}) in the case $H_{s_{1}-1, \de^{\pr}+\frac{3}{2}}(\Si)\times H_{s_{1}+1, \de_{1}-\frac{1}{2}}(\Si)\rightarrow H_{s_{1}-1, \de^{\pr}+\frac{1}{2}}(\Si)$, we only need to show,
$$\lim_{\tau\rightarrow 0}\|\partial_{t}f(x, u(x)+\tau v(x))-\partial_{t}f(x, u(x))\|_{H_{s_{1}-1, \de^{\pr}+\frac{3}{2}}(\Si)}=0.$$
This convergence follows from the continuity of $(\partial_{t}f, u)\rightarrow \partial_{t}f(x, u(x))$ as a map $H_{s_{1}, \de^{\pr}+1}(\Om_{\theta})\times H_{s_{1}+1, \de_{1}-\frac{1}{2}}(\Si)\rightarrow H_{s_{1}-1, \de^{\pr}+\frac{3}{2}}(\Si)$ in lemma \ref{composition at x u}. Now the multiplication operator $L_{u}: v\rightarrow \partial_{t}f(x, u(x))\cdot v$ is a bounded linear operator $L(H_{s_{1}+1, \de_{1}-\frac{1}{2}}(\Si), H_{s_{1}-1, \de^{\pr}+\frac{1}{2}}(\Si))$ with
$$\|L_{u}\|_{L(H_{s_{1}+1, \de_{1}-\frac{1}{2}}(\Si), H_{s_{1}-1, \de^{\pr}+\frac{1}{2}}(\Si))}\leq C\|\partial_{t}f(x, u(x))\|_{H_{s_{1}-1, \de^{\pr}+\frac{3}{2}}(\Si)}$$
by inequality (\ref{multiplication inequality}). The operator $L_{u}$ is also continuous w.r.t $u$ by lemma \ref{composition at x u}, so we know that $\F$ is Fr\'echet differentiable by Theorem 1.1.3 in \cite{Cha}, and $D_{u}\F(v)=\partial_{t}f(x, u(x))\cdot v$.
\end{proof}

Now we can prove the differentiability of $H_{u}$ w.r.t. $u$.
\begin{proposition}\label{differentiability w.r.t u}
For $s\geq 4$, $\de>-2$. Given a vacuum data $(g, k)\in\VC_{s+1, \de+\frac{1}{2}}(\Si)$ and $\theta$ the boost ratio as in the beginning of this section. If $\B_{\rho}\subset H_{s, \de-\frac{1}{2}}(\Si)$ with $\rho$ satisfying Conditions (A)(B) as in (\ref{condition (A)})(\ref{condition (B)}) for the $\theta$, then the mean curvature operator (\ref{mean curvature operator}) $\mH: \B_{\rho}\rightarrow H_{s-2, \de+\frac{3}{2}}(\Si)$ is continuous differentiable w.r.t. $u$, i.e. $(D_{u}\mH)\in C\big(\B_{\rho}, L(H_{s, \de-\frac{1}{2}}(\Si), H_{s-2, \de+\frac{3}{2}}(\Si))\big)$. Furthermore, $D_{u}\mH$ is given by the formal variational formula.
\end{proposition}
\begin{proof}
By the choice of $\rho$, $\mH$ is well-defined. Write out the expression for $H_{u}$ in (\ref{Hu1}) in local coordinates $\{(t, x^{i}):\ i=1, 2, 3\}$ of $\Om_{\theta}$ as follows:
\begin{equation}\label{Hu1 in local coordinates}
\begin{split}
H_{u} &=(g_{M})^{ij}\lan\nabla_{\al_{i}}N, \al_{j}\ran_{\ga}=\nu\cdot(g_{M})^{ij}\lan\nabla_{\partial_{i}+u_{i}\partial_{t}}(U+T), \partial_{j}+u_{j}\partial_{t}\ran_{\ga}\\
      &=\nu\cdot(g_{M})^{ij}\big\{(\partial_{i}+u_{i}\partial_{t})(U+T)^{\mu}\lan\partial_{\mu}, \partial_{j}+u_{j}\partial_{t}\ran_{\ga}+(U+T)^{\mu}\lan\nabla_{\partial_{i}+u_{i}\partial_{t}}\partial_{\mu}, \partial_{j}+u_{j}\partial_{t}\ran_{\ga}\big\}\\
      &=\nu\cdot(\ga^{ij}+\frac{\nu^{2}}{\al^{2}}(\be^{i}-\al U^{i})(\be^{j}-\al U^{j}))\big\{(\partial_{i}+u_{i}\partial_{t})(U^{\mu}+T^{\mu})\cdot(\ga_{\mu j}+u_{j}\ga_{\mu t})\\
      &+(U^{\mu}+T^{\mu})(\Ga_{i\mu, j}+u_{i}\Ga_{t\mu, j}+u_{j}\Ga_{i\mu, t}+u_{i}u_{j}\Ga_{t\mu, t})\big\},
\end{split}
\end{equation}
where $\Ga_{\mu\nu, \si}$ is the Christoffel symbol for $\ga$, and all coefficients of $\ga$ are evaluated at $(x, u(x))$. Except for the term $\nu$, $H_{u}$ is an algebraic expression containing two type of terms in ($\ref{Hu1 in local coordinates}$). One type of terms are the composition of the coefficients of $(\ga-\ti{\eta})$ and $\partial\ga$ with $(x, u(x))$, and the other terms contains $\partial u$ and $\partial^{2}u$. The only term appears in the denominator is $1+\lan\be, Du\ran_{g}$, and $|\lan\be, Du\ran_{g}|\leq\frac{1}{2}$ by the choice of $\rho$ as in Condition(B).

Since $(\ga-\ti{\eta})\in H_{s+1, \de}(\Om_{\theta})$, the composition of the metric coefficients of $(\ga-\ti{\eta})$ with $(x, u(x))$, i.e. $\{(\ga^{\mu\nu}-\ti{\eta}^{\mu\nu}),\ (\ga_{\mu\nu}-\ti{\eta}_{\mu\nu}),\ (\al-1), \be^{i}, \be_{i}, (g^{ij}-e^{ij}), (g_{ij}-e_{ij})\}(x, u(x))$ are continuous differentiable w.r.t. $u$ as maps $H_{s, \de-\frac{1}{2}}(\Si)\rightarrow H_{s-2, \de+\frac{1}{2}}(\Si)$ by lemma \ref{differentiable}. Similarly the composition of the coefficients of $\partial\ga$ with $(x, u(x))$, i.e. $(\partial\ga)(x, u(x))$ are also continuous differentiable w.r.t. $u$ as maps $H_{s, \de-\frac{1}{2}}(\Si)\rightarrow H_{s-2, \de+\frac{3}{2}}(\Si)$. The terms $\partial u$ and $\partial^{2}u$ are trivially continuous differentiable w.r.t. $u$ as maps $H_{s, \de-\frac{1}{2}}(\Si)\rightarrow H_{s-1, \de+\frac{1}{2}}(\Si)$ and $H_{s, \de-\frac{1}{2}}(\Si)\rightarrow H_{s-2, \de+\frac{3}{2}}(\Si)$ respectively. So $U=\frac{\al Du}{1+\lan\be, Du\ran}\in H_{s-1, \de+\frac{1}{2}}(\Si)$ and is continuous differentiable w.r.t. $u$, hence is $\nu^{2}-1=\frac{|U|^{2}}{1-|U|^{2}}\in H_{s-1, \de+\frac{3}{2}}(\Si)$, since $|U|\leq\frac{1}{2}$. So by similar argument as that for $\al$ in lemma \ref{sobolev bound for metric coefficient}, $(\nu-1)$ is also continuous differentiable w.r.t. $u$ as $H_{s, \de-\frac{1}{2}}(\Si)\rightarrow H_{s-2, \de+\frac{3}{2}}(\Si)$. Combing all them together, $H_{u}$ is continuously differentiable w.r.t. $u$ by the multiplication lemma \ref{properties of weighted sobolev case3}.
\end{proof}


\subsection{Linear theory}

Given a 3-dimensional manifold $(\Si, e)$ which is Euclidean at infinity. Let us give some results about linear elliptic operators which are asymptotic to the Laplacian $\lap_{e}$ on $(\Si, e)$. Such type of elliptic operators have been widely studied in \cite{Ba2}\cite{ChoCH}\cite{ChOM}\cite{L}.

Let $L$ be an operator on $(\Si, e)$ of form:
$$Lu=\Si_{k=0}^{2}a_{k}\partial^{k}u,$$
with $u$ and $Lu$ functions on $\Si$, satisfying:
\begin{equation}\label{hypothesis for L}
\begin{split}
& \la e\leq a_{2}\leq \la^{-1}e\ \textrm{as metrics, with $\la$ the \emph{elliptic coefficient};}\\
& (a_{2}-e)\in H_{s_{0}+1, \de_{0}}(\Si),\ a_{1}\in H_{s_{0}, \de_{0}+1}(\Si),\ a_{0}\in H_{s_{0}-1, \de_{0}+2}(\Si),
\end{split}
\end{equation}
where $s_{0}\geq 4$, $\de_{0}>-\frac{3}{2}$. We will show that in certain weighted spaces, such $L$ have uniformly bounded inverse on the orthogonal compliment of $ker(L)$ depending only on the norms of the coefficients. First we have,
\begin{lemma}\label{L2 estimate for L}
Given $s\leq s_{0}$, $-\frac{3}{2}<\de<-\frac{1}{2}$. There exists a constant $C$ and a large $r>R$, depending only on $s_{0}$, $\de_{0}$, the elliptic coefficient $\la$ and the norms $\|a_{2}-e\|_{H_{s_{0}, \de_{0}}(\Si)}$, $\|a_{1}\|_{H_{s_{0}-1, \de_{0}+1}(\Si)}$ and $\|a_{0}\|_{H_{s_{0}-2, \de_{0}+2}(\Si)}$, such that for any $u\in H_{s, \de-1}(\Si)$,
\begin{equation}\label{estimates for L}
\|u\|_{H_{s, \de-1}(\Si)}\leq C(\|Lu\|_{H_{s-2, \de+1}(\Si)}+\|u\|_{H_{s-2}(\Si_{int, 2r})}),
\end{equation}
where $\Si_{int, 2r}$ is the union of $\Si_{int}$ with all the annuli $B_{2r}\setminus B_{R}$ inside each end $\Si_{i}$, and $H_{s-2}$ is the standard $L^{2}$ Sobolev space on $\Si_{int, 2r}$.
\end{lemma}
\begin{proof}
Let $\Si=\Si_{int}\cup_{i=1}^{l}E_{i}$. Given a function $\chi\in C^{\infty}_{c}(\R^{3}\setminus B_{1})$, such that $0\leq\chi\leq 1$ and $\chi=1$ on $\R^{3}\setminus B_{2}$. We can find a partition of unity $\{\chi_{i, r}\}_{i=0}^{l}$ of $\Si$ for $r>R$, with $\chi_{i, r}(x)=\chi(|x|/r)$ for $x\in E_{i}\cong\R^{3}\setminus B_{R}$, and $\chi_{i, r}(x)=0$ for $x\in\Si\setminus E_{i}$, and $\chi_{0, r}(x)=1-\Si_{i=1}^{l}\chi_{i, r}(x)$. Then $u=\Si_{i=1}^{l}u_{i, r}$, with $u_{i, r}=\chi_{i, r}u$. Let us fix an end $E_{i}$ and $u_{i, r}$ and forget the sub-index $i$ now. Since $-\frac{3}{2}<\de<-\frac{1}{2}$ corresponds to non-exceptional value in \cite{Ba2}, we can apply Theorem 1.7 in \cite{Ba2} with $p=2$ here,
\begin{equation}\label{estimates for L step 1}
\begin{split}
\|u_{r}\|_{H_{s, \de-1}(\R^{3})} &\leq C_{1}\|\lap u_{r}\|_{H_{s-2, \de+1}(\R^{3})}\\
                                 &\leq C_{1}\big\{\|Lu_{r}\|_{H_{s-2, \de+1}(E)}+\|(L-\lap)u_{r}\|_{H_{s-2, \de+1}(E)}\big\},
\end{split}
\end{equation}
where $\lap$ is the laplacian operator w.r.t. $\de_{ij}$ and $C_{1}$ a uniform constant.
\begin{equation}
\begin{split}
\|Lu_{r}\|_{H_{s-2, \de+1}(E)} &\leq\|\chi_{r}Lu\|_{H_{s-2, \de+1}(E)} +\|2a_{2}^{ij}\partial_{i}u\partial_{j}\chi_{r}+(a_{2}^{ij}\partial^{2}\chi_{r}+a_{1}^{i}\partial\chi_{r})u\|_{H_{s-2, \de+1}(E)}\\
                               &\leq C_{2}(r)(\|Lu\|_{H_{s-2, \de+1}(E)}+\|u\|_{H_{s-1}(A_{r})}),
\end{split}
\end{equation}
with $A_{r}=B_{2r}\setminus B_{r}$, and $C_{2}(r)$ is a constant depending only on $r$ and $\|a_{2}-e\|_{H_{s_{0}, \de_{0}}(A_{r})}$, $\|a_{1}\|_{H_{s_{0}-1, \de_{0}+1}(A_{r})}$. Since $\de_{0}>-\frac{3}{2}$, we can assume $\de_{0}-\ep=\de_{1}>-\frac{3}{2}$ for some $\ep>0$. Using multiplication lemma \ref{properties of weighted sobolev case2},
\begin{equation}\label{estimates for L-lap on end}
\begin{split}
\|(L&-\lap)u_{r}\|_{H_{s-2, \de+1}(E_{r})}=\|(a_{2}^{ij}-\de^{ij})\partial^{2}u_{r}+a_{1}^{i}\partial u_{r}+a_{0}u_{r}\|_{H_{s-2, \de+1}(E_{r})}\\
    &\leq C_{3}\big(\|a_{2}-e\|_{H_{s, \de_{1}}(E_{r})}+\|a_{1}\|_{H_{s-1, \de_{1}+1}(E_{r})}+\|a_{0}\|_{H_{s-2, \de_{1}+2}(E_{r})}\big)\|u_{r}\|_{H_{s, \de-1}(E_{r})},
\end{split}
\end{equation}
where $E_{r}=\R^{3}\setminus B_{r}$ and $C_{3}$ a uniform constant. Now $\|a_{2}-e\|_{H_{s, \de_{1}}(E_{r})}+\|a_{1}\|_{H_{s-1, \de_{1}+1}(E_{r})}+\|a_{0}\|_{H_{s-2, \de_{1}+2}(E_{r})}\leq\big(\|a_{2}-e\|_{H_{s, \de_{0}}(E_{r})}+\|a_{1}\|_{H_{s-1, \de_{0}+1}(E_{r})}+\|a_{0}\|_{H_{s-2, \de_{0}+2}(E_{r})}\big)r^{-\ep}$ for $r$ large enough. So we can always choose a $r>R$, depending only on $\de_{0}$ and $\big(\|a_{2}-e\|_{H_{s, \de_{0}}(E_{r})}+\|a_{1}\|_{H_{s-1, \de_{0}+1}(E_{r})}+\|a_{0}\|_{H_{s-2, \de_{0}+2}(E_{r})}\big)$, such that $\|(L-\lap)u_{r}\|_{H_{s-2, \de+1}(E_{r})}\leq \frac{1}{2C_{1}}\|u_{r}\|_{H_{s, \de-1}(E_{r})}$. Putting them back to inequality (\ref{estimates for L step 1}),
\begin{equation}
\|u_{r}\|_{H_{s, \de-1}}\leq C_{4}\{\|Lu\|_{H_{s-2, \de+1}(E)}+\|u\|_{H_{s-1}(A_{r})}\},
\end{equation}
where $C_{4}$ depends only on $C_{2}(r)$. Using an interpolation inequality(see Lemma 2.2 in \cite{ChoCH}) to $\|u\|_{H_{s-1}(A_{r})}$, we can get the estimate of (\ref{estimates for L}) on each end. Applying the standard $L^{2}$ estimates to $u_{0, r}$ on $\Si_{int, 2r}$(See Corollary 2.2 on page 547 in \cite{Cho}),
\begin{equation}
\|u_{0, r}\|_{H_{s}(\Si_{int, 2r})}\leq C_{5}\{\|Lu_{0, r}\|_{H_{s-2}(\Si_{int, 2r})}+\|u_{0, r}\|_{H_{s-2}(\Si_{int, 2r})}\},
\end{equation}
where $C_{5}$ depends only on $s_{0}$, the elliptic coefficient $\la$ and the norms $\|a_{2}-e\|_{H_{s}(V_{int, 2r})}$, $\|a_{1}\|_{H_{s-1}(V_{int, 2r})}$ and $\|a_{0}\|_{H_{s-2}(V_{int, 2r})}$.
Combing results on all ends $E_{i, r}$ and $V_{int, 2r}$ together, we can get (\ref{estimates for L}) with $r$ and constant $C$ satisfying the requirement.
\end{proof}

Now we can prove a lemma similar to Theorem 1.10 in \cite{Ba2} and Theorem 5.6 in \cite{L}.
\begin{lemma}\label{Fredholmness of L}
Given $s\leq s_{0}$, $-\frac{3}{2}<\de<-\frac{1}{2}$, the operator $L$ is a Fredholm operator:
$$H_{s, \de-1}(\Si)\rightarrow H_{s-2, \de+1}(\Si),$$
i.e. $L$ has finite-dimensional kernel $ker(L, \de-1)=\{v\in H_{s, \de-1}(\Si):\ Lv=0\}$, and finite-dimensional co-kernel $coker(L, \de-1)$.
\end{lemma}
\begin{proof}
From the multiplication lemma \ref{properties of weighted sobolev case2}, we know that $L$ is a bounded linear map $H_{s, \de-1}(\Si)\rightarrow H_{s-2, \de+1}(\Si)$. Standard argument using inequality (\ref{estimates for L}) as in Theorem 1.10 in \cite{Ba2} shows that $N(L)$ is finite-dimensional and $L$ has close range. So $L$ is semi-Fredholm.

To show that $L$ has finite-dimensional co-kernel, we will borrow the techniques in Theorem 5.6 of \cite{L}. First, inequality (\ref{estimates for L-lap on end}) shows that the operator norm of $(L-\lap): H_{s, \de-1}(E_{r})\rightarrow H_{s-2, \de+1}(E_{r})$ is $o(1)$ as $r\rightarrow\infty$. So for large enough $r$, the fact that $\lap$ is Fredholm by Theorem 1.7 in \cite{Ba2} and that the Fredholm property is open w.r.t operator norms show that $L_{i}=\lap+\chi_{i, r}(L-\lap)$ is Fredholm on $\R^{3}$, where clearly $L_{i}=L$ on $E_{2r}$. So there exists a bounded linear operator $S_{i}: H_{s-2, \de+1}(\R^{3})\rightarrow H_{s, \de-1}(\R^{3})$, such that $L_{i}S_{i}=id+K_{i}$ with $K_{i}$ a compact operator. Now $L: H_{s, \de-1}(\Si_{int, 8r})\rightarrow H_{s-2, \de+1}(\Si_{int, 8r})$ is Fredholm since $\Si_{int, 8r}$ is compact, so there exists a Fredholm inverse $S_{0}: H_{s-2, \de+1}(\Si_{int, 8r})\rightarrow H_{s, \de-1}(\Si_{int, 8r})$, such that $LS_{0}=id+K_{0}$, for $K_{0}$ compact operator. Define
\begin{equation}
Su=\chi_{0, 4r}S_{0}u_{0, 8r}+\Si_{i=1}^{l}\chi_{i, 2r}S_{i}u_{i, r},
\end{equation}
which is a bounded linear operator $H_{s-2, \de+1}(\Si)\rightarrow H_{s, \de-1}(\Si)$. Then
a calculation as in (5.6.5) of \cite{L} shows that $LS=id+K$ for $K$ compact operator. So $L$ has finite-dimensional co-kernel.
\end{proof}

The Fredholm index of $L$ is defined to be:
$$i(L, \de-1)=dim\ ker(L, \de-1)-dim\ coker(L, \de-1).$$
By comparing the index of $L$ to the laplacian $\lap_{e}$ of $e$, we can show that $L$ is surjective when $a_{0}\leq 0$.
\begin{lemma}\label{surjectivity of L}
Given $s\leq s_{0}$, $-\frac{3}{2}<\de<-\frac{1}{2}$. Suppose $a_{0}\leq 0$, then $L$ is surjective. Furthermore, $dim\ ker(L, \de-1)=d_{\de-1}=dim\ ker(\lap_{e}, \de-1)$. If we denote $ker(L, \de-1)^{\perp}$ to be the orthogonal compliment of $ker(L, \de-1)$ w.r.t the $L^{2}_{\de-1}$ inner product $\lan\cdot, \cdot\ran_{L^{2}_{\de-1}(\Si)}$ as in definition \ref{def for weighted sobolev space2}, then:
$$L: ker(L, \de-1)^{\perp}\rightarrow H_{s-2, \de+1}(\Si),$$
is an isomorphism.
\end{lemma}
\begin{proof}
Since $L$ can be joint continuously to $\lap_{e}$ by $L_{t}=tL+(1-t)\lap_{e}$, we know $i(L, \de-1)=i(\lap_{e}, \de-1)$. Theorem 6.2 in \cite{L} says that $\lap_{e}$ is surjective when $\de-1<-\frac{1}{2}$. In order to show $L$ is surjective, or equivalently $dim\ coker(L, \de-1)=0$, we only need to show $dim\ ker(L, \de-1)\leq dim\ ker(\lap_{e}, \de-1)$. This comes from the asymptotical expansion given in \cite{Ba2}. For $u\in ker(L, \de-1)$, by Theorem 1.17 in \cite{Ba2}, $Lu=0$ implies that on each end $E_{i}$, there exists a harmonic homogenous function $h_{k}$ of order $k\leq k(\de)$, where $k(\de)=max\{k\in\mZ:\ k\leq -(\de+\frac{3}{2})\}$\footnote{See the definition for $k(\de)$ in \cite{Ba2}. Their $\de$ is the same as $-(\de+\frac{3}{2})$ here.}, such that $u=h_{k}+o(r^{k-\be})$ for $0<\be<\de+\frac{3}{2}$. In our case, $k(\de)=0$. In fact, if $u\neq 0$, there must exist at one end, on which $k\geq 0$. Or the decay implies $u=o(1)$ at infinity on $\Si$, so $u=0$ by maximum principle since $a_{0}\leq 0$. So $dim\ ker(L, \de-1)$ is less or equal to the the number of linearly independent harmonic polynomials of order $\leq k(\de)$ multiplied with the number of ends. It is easy to see that the basis of $ker(\lap_{e}, \de-1)$ is consisted just by functions which have main part the harmonic polynomial on one end, and $O(1/r)$ parts in other ends. So the leading terms shows that $dim\ ker(L, \de-1)\leq dim\ ker(\lap_{e}, \de-1)$. The isomorphism on orthogonal compliment is direct when $L$ is surjective.
\end{proof}

In fact, we can show a uniform norm bound for the inverse of $L$ on $ker(L, \de-1)^{\perp}$.
\begin{lemma}\label{inverse bound for L}
Given $s\leq s_{0}$, $-\frac{3}{2}<\de<-\frac{1}{2}$. Suppose $a_{0}\leq 0$. Denote the inverse of $L: ker(L, \de-1)^{\perp}\rightarrow H_{s-2, \de+1}(\Si)$ by $L^{-1}$, then there exists a constant $C$ depending only on $s_{0}$, $\de_{0}$, the elliptic coefficient $\la$ and the norms $\|a_{2}-e\|_{H_{s_{0}+1, \de_{0}}(\Si)}$, $\|a_{1}\|_{H_{s_{0}, \de_{0}+1}(\Si)}$ and $\|a_{0}\|_{H_{s_{0}-1, \de_{0}+2}(\Si)}$, such that for any $v\in H_{s-2, \de+1}(\Si)$,
\begin{equation}
\|L^{-1}v\|_{H_{s, \de-1}(\Si)}\leq C\|v\|_{H_{s-2, \de+1}(\Si)}.
\end{equation}
\end{lemma}
\begin{proof}
We only need to show that for any $u\in ker(L, \de-1)^{\perp}$,
$$\|u\|_{H_{s, \de-1}(\Si)}\leq C_{1}\|Lu\|_{H_{s-2, \de+1}(\Si)}$$
for a uniform constant $C_{1}$ depending only on $s_{0}$, $\de_{0}$, the elliptic coefficient $\la$ and the norms $\|a_{2}-e\|_{H_{s_{0}+1, \de_{0}}(\Si)}$, $\|a_{1}\|_{H_{s_{0}, \de_{0}+1}(\Si)}$ and $\|a_{0}\|_{H_{s_{0}-1, \de_{0}+2}(\Si)}$. By contradiction argument, suppose that the statement is wrong, which means that there exists a sequence of operators $L_{i}$ with $a_{i, 0}\leq 0$, uniformly bounded elliptic coefficient $\la_{i}\geq\la_{0}>0$ and uniformly bounded coefficients $\|a_{i, 2}-e\|_{H_{s_{0}+1, \de_{0}}(\Si)}$, $\|a_{i, 1}\|_{H_{s_{0}, \de_{0}+1}(\Si)}$, $\|a_{i, 0}\|_{H_{s_{0}-1, \de_{0}+2}(\Si)}\leq C_{0}$, and a sequence of functions $u_{i}\in ker(L_{i}, \de-1)^{\perp}$, such that $\|u_{i}\|_{H_{s, \de-1}(\Si)}\geq i\|L_{i}u_{i}\|_{H_{s-2, \de+1}(\Si)}$. By re-normalizing, we get a sequence of functions $u_{i}$, with $\|u_{i}\|_{H_{s, \de-1}(\Si)}=1$, while $\|L_{i}u_{i}\|_{H_{s-2, \de+1}(\Si)}\rightarrow 0$. By weak compactness, there exists a subsequence, which we still denote by $L_{i}$, such that the coefficients of $L_{i}$ converges weakly to that of a linear operator $L_{\infty}$ with $\la_{0}e\leq a_{\infty, 2}\leq\la_{0}^{-1}e$, $a_{\infty, 0}\leq 0$ and $\|a_{\infty, 2}-e\|_{H_{s_{0}+1, \de_{0}}(\Si)}$, $\|a_{\infty, 1}\|_{H_{s_{0}, \de_{0}+1}(\Si)}$, $\|a_{\infty, 0}\|_{H_{s_{0}-1, \de_{0}+2}(\Si)}\leq C_{0}$. Using inequality (\ref{estimates for L}), there is a uniform constant $C_{2}$,
\begin{equation}\label{ui-uj}
\begin{split}
\|u_{i}-u_{j}\|_{H_{s, \de-1}(\Si)} &\leq C_{2}(\|L_{i}(u_{i}-u_{j})\|_{H_{s-2, \de+1}(\Si)}+\|u_{i}-u_{j}\|_{H_{s-2}(\Si_{int, 2r})})\\
                                    &\leq C_{2}(\|L_{i}u_{i}\|_{H_{s-2, \de+1}(\Si)}+\|(L_{i}-L_{j})u_{j}\|_{H_{s-2, \de+1}(\Si)}\\
                                    &+\|L_{j}u_{j}\|_{H_{s-2, \de+1}(\Si)}+\|u_{i}-u_{j}\|_{H_{s-2}(\Si_{int, 2r})}).
\end{split}
\end{equation}
Now $\|(L_{i}-L_{j})u_{j}\|_{H_{s-2, \de+1}(\Si)}\leq C(\|a_{2, i}-a_{2, j}\|_{H_{s_{0}, \de^{\pr}-1}(\Si)}+\|a_{1, i}-a_{1, j}\|_{H_{s_{0}-1, \de^{\pr}}(\Si)}+\|a_{0, i}-a_{0, j}\|_{H_{s_{0}-2, \de^{\pr}+1}(\Si)})$ for some $\de_{0}>\de^{\pr}>-\frac{3}{2}$ by multiplication lemma \ref{properties of weighted sobolev case3}. The compact embedding (Lemma 2.1 in \cite{ChoCH}) of $H_{s_{0}+1-i, \de_{0}-1+i}(\Si)\subset H_{s_{0}-i, \de^{\pr}-1+i}(\Si)$ for $i=0,1,2$ imply that $\|(L_{i}-L_{j})u_{j}\|_{H_{s-2, \de+1}(\Si)}\rightarrow 0$ for a subsequence of $\{L_{i}\}$. Together with the compactness of $H_{s, \de-1}(\Si)\subset H_{s-2}(\Si_{int, 2r})$, there exists a subsequence, which we still denote by $u_{i}$, such that $u_{i}$ converge strongly in $H_{s, \de-1}(\Si)$ to a function $u_{\infty}$, with $\|u_{\infty}\|_{H_{s, \de-1}(\Si)}=1$. Furthermore we have that $L_{\infty}u_{\infty}=0$ weakly by the weak convergence, and hence strongly in $H_{s-2, \de+1}(\Si)$ by elliptic regularity.

By Lemma \ref{surjectivity of L}, we know that $dim ker(L_{i}, \de-1)\equiv d_{\de-1}$. We claim that $ker(L_{i}, \de-1)$ converge to a $d_{\de-1}$ dimensional linear subspace of $ker(L_{\infty}, \de-1)$. Let $\{v_{i, a}\}_{a=1}^{d_{\de-1}}$ be an $L^{2}_{\de-1}$ orthogonal basis for $ker(L_{i}, \de-1)$, with $\|v_{i, a}\|_{H_{s, \de-1}(\Si)}=1$. By equation (\ref{ui-uj}),
$$\|v_{i, a}-v_{j, a}\|_{H_{s, \de-1}(\Si)}\leq C(\|(L_{i}-L_{j})v_{j, a}\|_{H_{s-2, \de+1}(\Si)}+\|v_{i, a}-v_{j, a}\|_{H_{s-2}(\Si_{int, 2r})}).$$
Similar argument as above implies that a subsequence of $v_{i, a}$ converge strongly in $H_{s, \de-1}(\Si)$ to some $v_{\infty, a}$. Hence $v_{\infty, a}\in ker(L_{\infty}, \de-1)$, and $\{v_{\infty, a}\}_{a=1}^{d_{\de-1}}$ are also orthogonal in $L^{2}_{\de-1}$ with $\|v_{\infty, a}\|_{H_{s, \de-1}(\Si)}=1$. Since $L_{\infty}$ satisfies all the requirement of Lemma \ref{surjectivity of L}, $dim ker(L_{\infty}, \de-1)=d_{\de-1}$. Hence the limit of $ker(L_{i}, \de-1)$ is exactly the entire $ker(L_{\infty}, \de-1)$. As $u_{i}$ is perpendicular to $ker(L_{i}, \de-1)$ in $L^{2}_{\de-1}$, passing to the limit, we know that $u_{\infty}$ is perpendicular to $ker(L_{\infty}, \de-1)$ in $L^{2}_{\de-1}$ too, which is a contradiction to that $\|u_{\infty}\|_{H_{s, \de-1}(\Si)}=1$ and $L_{\infty}u_{\infty}=0$. So we finish the proof.
\end{proof}


\subsection{Existence of maximal data}

Now let us calculate the linearization of $\mH$ with respect to $u$ at $(g, k, 0)$. Fix a vacuum data $(g, k)\in\VC_{s+2, \de+\frac{1}{2}}(\Si)$ with the unique boost solution $(\V, \ga)$ given by Theorem \ref{boost theorem}. Recall the form (\ref{metric form for gamma}) of $\ga$ in local coordinates $(x^{i}, t)$ of $\Om_{\theta}$. According to the initial data equations (\ref{initial data1})(\ref{initial data2}) for $\ga$, the coefficients restricted to $t=0$ slice are given by:
\begin{equation}
\al|_{\Si}\equiv 1;\ \ \be|_{\Si}\equiv 0.
\end{equation}
In fact, our choice of $\al|_{\Si}$ and $\be|_{\Si}$ implies that $\partial_{t}|_{\Si}$ is the unit normal of $\Si$. Now recall the second variational formula for the mean curvature in section 2 of \cite{Ba1}. Let $X$ be a vector field in a neighborhood of $\Si$ with associated flow $\phi_{s}: \V\rightarrow\V$. Denote $H(s)$ by the mean curvature of $\phi_{s}(\Si)$, then
\begin{equation}\label{2nd variation of mean curvature}
\partial_{s}(H(s))|_{s=0}=-\lap_{g}\lan X, N\ran+\lan X, N\ran(|k|_{g}^{2}+Ric_{\ga}(N, N))+\lan X, \nabla_{g}H\ran,
\end{equation}
where $N$ is the unit normal of $\Si$ insider $\V$, and $Ric_{\ga}$ the Ricci curvature of $\ga$. In our case, $Ric_{\ga}\equiv 0$ by (\ref{EVE}) since our $(\V, \ga)$ is vacuum, and the unit normal $N=\partial_{t}$ on $\Si$. We can choose the vector field to be $X=v\partial_{t}$, where $v$ is a compactly supported smooth scalar function, so $\lan X, \nabla_{g}H\ran=0$. Then $\partial_{s}H(s)|_{s=0}$ is the linearization of $H$ w.r.t $u$, and $\lan X, N\ran=-v$. Now combining all and using Proposition \ref{differentiability w.r.t u}, we have,
\begin{lemma}\label{first variation for Hu at maximal data}
Using notations in Proposition \ref{differentiability w.r.t u}, the Fr\'echet derivative of $\mH(g, k, u)$ with respect to factor $u$ at a vacuum data $(g, k, 0)$ is a linear operator $L_{0}: H_{s, \de-\frac{1}{2}}(\Si)\rightarrow H_{s-2, \de+\frac{3}{2}}(\Si)$ given by:
\begin{equation}
(D_{v}\mH)_{(g, k, 0)}=L_{0}v=(\lap_{g}-|k|_{g}^{2})v.
\end{equation}
\end{lemma}

Now let us focus on the operator $L_{0}$. $L_{0}$ is in fact Fredholm and surjective by Lemma \ref{Fredholmness of L} and Lemma \ref{surjectivity of L}. By making use the fact that $L_{0}$ has finite-dimensional kernel and is surjective, we can get the existence of solutions of $\mH(g, k, u)=0$ for $(g, k)$ with small trace $tr_{g}k$ by a perturbation method, but no uniqueness due to the existence of non-trivial kernel $ker(L_{0}, \de-\frac{1}{2})$. We will give an existence and uniqueness theorem in the orthogonal compliment of the kernel in order to find symmetry preserving solutions in the following section. Let us first give a Quantitative Inverse Function Theorem motivated by \cite{R}.
\begin{theorem}\label{quantitative inverse function theorem}
Let $X$, $Y$ be Banach spaces, and $U\subset X$ an open set. Suppose $F: U\rightarrow Y$ is a continuous map, and has Fr\'echet derivative w.r.t $x$, such that $\frac{\partial F}{\partial x}$ is continuous. For a point $x_{0}\in U$, with $F(x_{0})=y_{0}$. Suppose $\frac{\partial F}{\partial x}(x_{0}): X\rightarrow Y$ is invertible, and $\|\big[\frac{\partial F}{\partial x}(x_{0})\big]^{-1}\|\leq C$. Assume that we can find $r_{0}>0$, such that for any $x\in B_{r_{0}}(x_{0})\subset U$,
\begin{equation}\label{condition for IFT}
\|\frac{\partial F}{\partial x}(x)-\frac{\partial F}{\partial x}(x_{0})\|\leq\frac{1}{2C}.
\end{equation}
Then for any $y\in Y$ with
$$|y-y_{0}|_{Y}<\frac{r_{0}}{2C},$$
there exist a unique $x\in B_{r_{0}}(x_{0})$, such that $F(x)=y$.
\end{theorem}
\begin{proof}
Fix a $y\in B_{r_{0}/2C}(y_{0})\subset Y$. Let us consider the map $T: B_{r_{0}}(0)\subset X\rightarrow Y$, defined by
$$T(x)=x-[\frac{\partial F}{\partial x}(x_{0})\big]^{-1}(F(x_{0}+x)-y).$$
$x$ is a fixed point if and only if $F(x_{0}+x)=y$. So let us use the fixed point theorem to find a fixed point for $T$ on $B_{r_{0}}(0)$. first, for any $x_{1}, x_{2}\in B_{r_{0}}(0)$,
\begin{equation}
\begin{split}
|T(x_{1})-T(x_{2})|_{X} &=|(x_{1}-x_{2})-[\frac{\partial F}{\partial x}(x_{0})\big]^{-1}(F(x_{0}+x_{1})-F(x_{0}+x_{2}))|_{X}\\
                        &\leq\|[\frac{\partial F}{\partial x}(x_{0})\big]^{-1}\|\cdot|\frac{\partial F}{\partial x}(x_{0})(x_{1}-x_{2})-\frac{\partial F}{\partial x}(x_{0}+\bar{x})(x_{1}-x_{2})|_{Y}\\
                        &\leq C\|\frac{\partial F}{\partial x}(x_{0})-\frac{\partial F}{\partial x}(x_{0}+\bar{x})\|\cdot|x_{1}-x_{2}|_{X}\\
                        &\leq C\frac{1}{2C}|x_{1}-x_{2}|_{X}\leq\frac{1}{2}|x_{1}-x_{2}|_{X},
\end{split}
\end{equation}
where we used the mean value theorem in the first $``\leq"$, and condition (\ref{condition for IFT}) in the third $``\leq"$. So $T$ is a contraction map on $B_{r_{0}}(0)$. Next, for any $x\in B_{r_{0}}(0)$, and $|y-y_{0}|_{Y}<\frac{r_{0}}{2C}$,
\begin{equation}
\begin{split}
|T(x)|_{X} & \leq\|[\frac{\partial F}{\partial x}(x_{0})\big]^{-1}\|\cdot|\frac{\partial F}{\partial x}(x_{0})x-(F(x_{0}+x)-F(x_{0}))-(y-F(x_{0}))|_{Y}\\
           & \leq C(|\big(\frac{\partial F}{\partial x}(x_{0})-\frac{\partial F}{\partial x}(x_{0}+\bar{x})\big)x|_{X}+|y-F(x_{0})|_{Y})\\
           & \leq C(\|\frac{\partial F}{\partial x}(x_{0})-\frac{\partial F}{\partial x}(x_{0}+\bar{x})\|\cdot|x|_{X}+|y-F(x_{0})|_{Y})\\
           & < C(\frac{1}{2C}r_{0}+\frac{r_{0}}{2C}) < r_{0},
\end{split}
\end{equation}
where we use condition (\ref{condition for IFT}) in the last $``<"$. So $T$ maps $B_{r_{0}}(0)$ to $B_{r_{0}}(0)$. By applying the Contraction Mapping Theorem to $T: B_{r_{0}}(0)\rightarrow B_{r_{0}}(0)$, we finish the proof.
\end{proof}
\begin{remark}
This can be viewed as a carefully reworking of the proof of Theorem 1.2.1 in \cite{Cha}. Theorem 3.1 and Theorem 3.2 in \cite{R} also gave a proof about the quantitative inverse function theorem.
\end{remark}

\begin{theorem}\label{main theorem first part}
For $s\geq 4$, $-2<\de<-1$. Fix a 3-manifold $(\Si, e)$ which is Euclidean at infinity and a $\la>0$. Given a vacuum data $(g, k)\in\VC_{s+2, \de+\frac{1}{2}}(\Si)$, with $g\geq \la e$, there is an $\ep>0$ and a $\rho^{\pr}>0$ small enough, depending only on the norms $\|g-e\|_{H_{s+2, \de+\frac{1}{2}}(\Si)}+\|k\|_{H_{s+1, \de+\frac{3}{2}}(\Si)}$ and the elliptic constant $\la$, such that if $\|tr_{g}k\|_{H_{s-2, \de+\frac{3}{2}}(\Si)}\leq\ep$, there exists a unique function $u\in ker(L_{0}, \de-\frac{1}{2})^{\perp}$ with $\|u\|_{H_{s, \de-\frac{1}{2}}(\Si)}\leq\rho^{\pr}$, such that $u$ is a solution of the maximal surface equation (\ref{maximal surface equation}).
\end{theorem}
\begin{proof}
For the given $(g, k)\in\VC_{s+2, \de+\frac{1}{2}}(\Si)$ with $\theta$ the boost ratio, we can choose a $\rho$ ball $\B_{\rho}\subset H_{s, \de-\frac{1}{2}}(\Si)$, with $\rho$ small enough depending only on $\theta$, $\|g-e\|_{H_{s+2, \de+\frac{1}{2}}}+\|k\|_{H_{s+1, \de+\frac{3}{2}}}$ and $\la$ as in Proposition \ref{differentiability w.r.t u}. Then the map $\mH$ is continuously differentiable w.r.t. $u$ as a map $\B_{\rho}\cap ker(L_{0}, \de-\frac{1}{2})^{\perp}\rightarrow H_{s-2, \de+\frac{3}{2}}(\Si)$, and the Fr\'echet derivative is $(D_{u}\mH)_{(g_, k, 0)}=L_{0}v=(\lap_{g}-|k|_{g}^{2})v$ by Lemma \ref{first variation for Hu at maximal data}. The coefficient of $L_{0}$ satisfies the hypothesis (\ref{hypothesis for L}), where $s_{0}=s+1$ and $\de_{0}=\de+\frac{1}{2}$ by multiplication lemma \ref{properties of weighted sobolev case2}, the elliptic constant equals to $\la$ and $\|a_{0, 2}-e\|_{H_{s+2, \de+\frac{1}{2}}(\Si)}$, $\|a_{0, 1}\|_{H_{s+1, \de+\frac{3}{2}}(\Si)}$, $\|a_{0, 0}\|_{H_{s, \de+\frac{5}{2}}(\Si)}$ are bounded from above by a constant depending only on $\|g-e\|_{H_{s+2, \de+\frac{1}{2}}(\Si)}$ and $\|k\|_{H_{s+1, \de+\frac{3}{2}}(\Si)}$. So $(D_{u}\mH)_{(g, k, 0)}$ is an isomorphism $ker(L_{0}, \de-\frac{1}{2})^{\perp}\rightarrow H_{s-2, \de+\frac{3}{2}}(\Si)$ by Lemma \ref{surjectivity of L}, since $a_{0, 0}=-|k|_{g}^{2}\leq 0$. Now we will show that the conditions in the quantitative Inverse Function Theorem \ref{quantitative inverse function theorem} are satisfied. By Lemma \ref{inverse bound for L}, there exists a constant $C_{0}$ depending only on $\la$, $\|a_{0, 2}-e\|_{H_{s+2, \de+\frac{1}{2}}(\Si)}$, $\|a_{0, 1}\|_{H_{s+1, \de+\frac{3}{2}}(\Si)}$, $\|a_{0, 0}\|_{H_{s, \de+\frac{5}{2}}(\Si)}$, such that,
$$\|L_{0}^{-1}\|_{L(H_{s-2, \de+\frac{3}{2}}(\Si),\ ker(L_{0}, \de-\frac{1}{2})^{\perp})}\leq C_{0}.$$
Abbreviate the operator norm $\|\cdot\|_{L(H_{s, \de-\frac{1}{2}}(\Si),\ H_{s-2, \de+\frac{3}{2}}(\Si))}=\|\cdot\|$.
Let us study $\|D_{u}\mH(g, k, u)-D_{u}\mH(g, k, 0)\|$. Fix the boost evolution $(\Om_{\theta}, \ga)$ of $(g, k)$, with $\|\ga-\ti{\eta}\|_{H_{s+2, \de}(\Om_{\theta})}$ uniformly bounded by a constant depending only on $\la$ and $\|g-e\|_{H_{s+2, \de+\frac{1}{2}}(\Si)}+\|k\|_{H_{s+1, \de+\frac{3}{2}}(\Si)}$. Then $D_{u}\mH(g, k, u)$ is the first variation $D_{u}(H_{u})$ of $H_{u}$ w.r.t. $u$ inside $(\Om_{\theta}, \ga)$. From the formula of $H_{u}$ in (\ref{Hu1 in local coordinates}), we know that $D_{u}(H_{u})$ is a second order differential operator. The coefficients of $D_{u}(H_{u})$ are constituted by algebraic expressions of $\partial u$, $\partial^{2}u$ and components of $\ga$, $\partial\ga$, $\partial^{2}\ga$ evaluated at $(x, u(x))$. Let $a$ be any component of $\partial^{2}\ga$(similar for $\ga$ and $\partial\ga$), using the Newton-Leibniz formula,
$$a(x, u(x))-a(x, 0)=\big(\int_{\tau=0}^{1}\partial_{t}a(x, \tau u(x))d\tau\big)u(x),$$
where $\partial_{t}a(x, u(x))$ has uniform $H_{s-2, \de+\frac{7}{2}}(\Si)$ norm depending only on $\|\partial^{3}\ga\|_{H_{s-1, \de+3}(\Om_{\theta})}$ and $\rho$ by Lemma \ref{composition at x u}. So $\|a(x, u(x))-a(x, 0)\|_{H_{s-2, \de+\frac{5}{2}}(\Si)}\leq C_{3}\|u\|_{H_{s, \de-\frac{1}{2}}(\Si)}$ by multiplication lemma \ref{properties of weighted sobolev case3}, where $C_{3}$ depends only on $\|\ga-\ti{\eta}\|_{H_{s+2, \de}(\Om_{\theta})}$ and $\rho$. Hence by comparing the coefficients of $D_{u}\mH(g, k, u)$ and $D_{u}\mH(g, k, 0)$, we can choose $\|u\|_{H_{s, \de-\frac{1}{2}}(\Si)}\leq\rho^{\pr}$ with $\rho^{\pr}$ small enough, depending only on $\|\ga-\ti{\eta}\|_{H_{s+2, \de}(\Om_{\theta})}$ and $C_{0}$, such that,
\begin{equation}
\|D_{u}\mH(g, k, u)-D_{u}\mH(g, k, 0)\|\leq\frac{1}{2 C_{0}}.
\end{equation}
For the $\rho^{\pr}$ chosen above, if we take $\ep<\frac{\rho^{\pr}}{2C_{0}}$, then
$$\|0-\mH(g, k, 0)\|_{H_{s-2, \de+\frac{3}{2}}(\Si)}=\|tr_{g}k\|_{H_{s-2, \de+\frac{3}{2}}(\Si)}<\frac{\rho^{\pr}}{2 C_{0}}.$$
Now by the Quantitative Inverse Function Theorem \ref{quantitative inverse function theorem}, if we choose the $\ep$ and $\rho^{\pr}$ as above, where $\ep$ and $\rho^{\pr}$ depend only on $\la$, $\|g-e\|_{H_{s+2, \de+\frac{1}{2}}(\Si)}$ and $\|k\|_{H_{s+1, \de+\frac{3}{2}}(\Si)}$, there exists a unique $u\in\B_{\rho^{\pr}}\cap ker(L_{0}, \de-\frac{1}{2})^{\perp}$, such that $u$ solves $\mH(u)=0$.
\end{proof}


\subsection{Proof of the main Theorems}\label{Proof of main theorem}

Here we will study the properties of the maximal graph gotten above. We will improve the regularity of the solution using a bootstrap argument, and show that the ADM mass of the maximal graph is the same as the given data. Moreover the maximal graph can be chosen to be axisymmetric if $(g, k)$ is axisymmetric, and the angular momentum of the maximal graph is the same as $(g, k)$.

In Theorem \ref{main theorem first part}, the solution $u$ has only $s$ weak derivatives due to the contraction mapping principal. In fact, by exploring the structure of the mean curvature operator (\ref{Hu1 in local coordinates}), we can gain more regularity for $u$.

\begin{lemma}\label{regularity of the maximal surface}
\emph{(Regularity analysis)}. In Theorem \ref{main theorem first part}, the solution $u\in H_{s+2, \de-\frac{1}{2}}(\Si)$. Denote $M=Graph_{u}$, and let $g_{M}$ be the metric and $k_{M}$ the second fundamental form induced by $M\subset(\V, \ga)$, then $(g_{M}, k_{M})\in\VC_{s+1, \de+\frac{1}{2}}(\Si)$.
\end{lemma}
\begin{proof}
In the local coordinates formula (\ref{Hu1 in local coordinates}), we can collect together all the terms containing $\partial^{2}_{ij}u$, then the maximal surface equation $\mH(u)=0$ can be rewritten as a linear second order elliptic equation for $u$ with $\partial u$ and $u$ terms as coefficients:
$$(g^{M})^{ij}(x, u(x))u_{ij}=f(x),$$
where $f(x)$ is a polynomial of $g^{M}(x, u(x))$, $\partial u$, $\ga(x, u(x))$ and $(\partial\ga)(x, u(x))$. First the spacelike property of $M=Graph_{u}$ implies that $(g^{M})^{ij}$ is elliptic. Furthermore, $(g^{M})^{ij}(x, u(x))-e^{ij}(x)=\ga^{ij}-e^{ij}+\frac{\nu^{2}}{\al^{2}}(\be^{i}-\al U^{i})(\be^{j}-\al U^{j})\in H_{s-1, \de+\frac{1}{2}}(\Si)$, $f(x)\in H_{s-1, \de+\frac{3}{2}}(\Si)$ by the argument in the proof of Proposition \ref{differentiability w.r.t u}, Lemma \ref{composition at x u} and the Banach algebra property in Lemma \ref{properties of weighted sobolev case3}. Since $(g^{M})^{ij}$ lie in $C^{0}$ and $H_{s-1}$ locally, $u\in (H_{s+1})_{loc}(\Si)$ by standard elliptic regularity theory. Furthermore, the linear operator $Lu=(g^{M})^{ij}\partial^{2}_{ij}u$ satisfies the hypothesis of the weighted elliptic regularity Theorem 6.1 in \cite{ChoCH} since $s\geq 4$, hence $u\in H_{s+1, \de-\frac{1}{2}}(\Si)$ by Theorem 6.1 in \cite{ChoCH}. Now we can bootstrap this process. In fact, by the composition Lemma \ref{composition at x u}, the right hand side $f(x)$ lies in at most $H_{s, \de+\frac{3}{2}}(\Si)$ since there are $\partial\ga(x, u(x))$ terms. So bootstrap ends when $u\in H_{s+2, \de-\frac{1}{2}}(\Si)$. 

On the graph $M$, $(g_{M})_{ij}=(g_{ij}+\be_{i}u_{j}+\be_{j}u_{i}-(\al^{2}-\be^{2})u_{i}u_{j})(x, u(x))$ by (\ref{graph metric}),
\begin{equation}\label{ku}
\begin{split}
(k_{M})_{ij} &=\nu\cdot\big\{(\partial_{i}+u_{i}\partial_{t})(U^{\mu}+T^{\mu})\cdot(\ga_{\mu j}+u_{j}\ga_{\mu t})\\
             &+(U^{\mu}+T^{\mu})(\Ga_{i\mu, j}+u_{i}\Ga_{t\mu, j}+u_{j}\Ga_{i\mu, t}+u_{i}u_{j}\Ga_{t\mu, t})\big\},
\end{split}
\end{equation}
by formula (\ref{Hu1 in local coordinates}). So by the proof of Proposition \ref{differentiability w.r.t u}, $((g_{M})_{ij}-e_{ij})\in H_{s+1, \de+\frac{1}{2}}(\Si)$ and $(k_{M})_{ij}\in H_{s, \de+\frac{3}{2}}(\Si)$.
\end{proof}

In order to define the ADM mass and linear momentum, we need to assume $-\frac{3}{2}<\de<-1$, then by the embedding lemma \ref{properties of weighted sobolev case2}, $(g_{M}-e)\in C^{s-1}_{\be}(\Si)$ and $k_{M}\in C^{s-2}_{\be+1}(\Si)$ for some $\frac{1}{2}<\be<1$, which satisfy the conditions (\ref{decay of g and k}). Similar conditions are also satisfied by $(g-e, k)$. We can defined the ADM mass $m$, $m_{M}$ for $(g, k)$ and $(g_{M}, k_{M})$ respectively.
\begin{lemma}\label{conservation of mass}
For $-\frac{3}{2}<\de<-1$, in Theorem \ref{main theorem first part}, $m=m_{M}$.
\end{lemma}
\begin{proof}
We will use the multiplication lemma \ref{properties of weighted sobolev case2} frequently when we multiply two Soblev functions. $(g_{M})_{ij}(x)-g_{ij}(x, u(x))=(\be_{i}u_{j}+\be_{j}u_{i}-(\al^{2}-\be^{2})u_{i}u_{j})(x, u(x))$ by (\ref{graph metric}). Now $\be(x, u(x))\in H_{s+1, \de+\frac{1}{2}}(\Si)$ and $\partial u\in H_{s+1, \de+\frac{1}{2}}(\Si)$ imply $(g_{M})_{ij}(x)-g_{ij}(x, u(x))\in H_{s+1, \de+1}(\Si)$. 
$$g_{ij}(x, u(x))-g_{ij}(x)=\{\int_{s=0}^{1}\partial_{t} g_{ij}(x, su(x))ds\}\cdot u(x),$$
which shows $\{g_{ij}(x, u(x))-g_{ij}(x)\}\in H_{s+1, \de+1}(\Si)$, since $\partial_{t}g_{ij}(x, su(x))\in H_{s+1, \de+\frac{3}{2}}(\Si)$ and $u\in H_{s+2, \de-\frac{1}{2}}(\Si)$. Hence $\{(g_{M})_{ij}(x)-g_{ij}(x)\}\in H_{s+1, \de+1}(\Si)\subset C^{s-1}_{\be}(\Si)$, for some $1<\be<\de+\frac{5}{2}$ by the embedding lemma \ref{properties of weighted sobolev case2}. By checking the definition (\ref{ADM mass}), we know that a error term of decay rate $o(r^{-1})$ will not change the mass, so $m=m_{M}$.
\end{proof}

Now we will study the preservation of symmetry by this constructions. We need a lemma about symmetry preserving by the reduced EVE (\ref{reduced EVE}).
\begin{lemma}\label{symmetry preserving1}
Given a vacuum data $(g, k)\in\VC_{s+2, \de+\frac{1}{2}}(\Si)$, and $(\Om_{\theta}, \ga)$ the boost evolution of $(g, k)$ given by Theorem \ref{boost theorem}. Suppose that both $(g, k)$ and $e$ are symmetric under a Killing vector field $\xi$ on $\Si$, i.e. $(g, k)$ satisfy (\ref{symmetry}), and $\mL_{\xi}e=0$, where $e$ is the canonical metric on $\Si$. Then the parallel translation $\ti{\xi}$ of $\xi$ into $\Om_{\theta}$ is a Killing vector field of $\ga$.
\end{lemma}
\begin{proof}
Now let $\phi_{s}: \Si\rightarrow\Si$ be the one parameter group of diffeomorphisms corresponding to $\xi$. Then $(\phi_{s})^{*}g=g$, $(\phi_{s})^{*}k=k$ and $(\phi_{s})^{*}e=e$. Now let us extend $\phi_{s}$ to a diffeomorphism $\ti{\phi}_{s}:\Om_{\theta}\rightarrow\Om_{\theta}$ by
\begin{equation}\label{ti-phi}
\ti{\phi}_{s}(x, t)=(\phi_{s}(x), t).
\end{equation}
Then $(\ti{\phi}_{s})^{*}\ti{e}=\ti{e}$ where $\ti{e}$ is defined by (\ref{e tilde}). By the initial conditions (\ref{initial data1})(\ref{initial data2}) for $\ga$, we know that $\ga_{s}=(\ti{\phi}_{s})^{*}\ga$ has the same initial conditions as $\ga$ on $\Si$. If we can show that $\ga_{s}$ also solves the reduced (EVE) (\ref{reduced EVE}), the uniqueness in Theorem \ref{boost theorem} implies that $\ga_{s}=\ga$. Since $\ga_{s}$ is Ricci flat, we only need to show that $(\Om_{\theta}, \ga_{s})$ is also in a harmonic gauge, or equivalently, $id: (\Om_{\theta}, \ga_{s})\rightarrow(\Om_{\theta}, \ti{e})$ is a wave map. By pulling back the wave map equation $\Box_{(\ga, \ti{e})}id=0$ by $\ti{\phi}$, we get $\Box_{((\ti{\phi}_{s})^{*}\ga, (\ti{\phi}_{s})^{*}\ti{e})}id=0$, which reduces to $\Box_{(\ga_{s}, \ti{e})}id=0$. This means that $\ga_{s}$ is also in a harmonic gauge, hence $\ga_{s}=\ga$. Now the vector field corresponding to $\ti{\phi}_{s}$ is clearly the parallel translation of $\xi$ into $\Om_{\theta}$.
\end{proof}

Now we can prove the preservation of symmetry for the maximal surface.
\begin{theorem}\label{main theorem second part}
Given $s\geq 4$, $-2<\de<-1$. Suppose $(\Si, e)$ is a 3-manifold, which is Euclidean at infinity and axisymmetric in the sense of Definition \ref{definition of axisymmetric}. If $(g, k)\in\VC^{a}_{s+2, \de+\frac{1}{2}}(\Si)$ is axisymmetric, and $\|tr_{g}k\|_{H_{s-2, \de+\frac{3}{2}}(\Si)}\leq\ep$ with $\ep$ given by Theorem \ref{main theorem first part}, then the solution $u$ of the maximal surface equation (\ref{maximal surface equation}) given in Theorem \ref{main theorem first part} can be chosen to be axisymmetric, i.e. $\partial_{\varphi}u=0$. Hence $(\Si, g_{u}, k_{u})$ is axisymmetric, and the angular momentum of $(g_{u}, k_{u})$ equals that of $(g, k)$.
\end{theorem}
\begin{proof}
By Theorem \ref{main theorem first part}, $\mH(g, k, u)=0$ has a unique solution $u\in\B_{\rho^{\pr}}\cap ker(L_{0}, \de-\frac{1}{2})^{\perp}$. Let $\phi_{s}$ be the diffeomorphism corresponding to the Killing vector field $\xi=\frac{\partial}{\partial\varphi}$ in Definition \ref{definition of axisymmetric}, and $\ti{\phi}_{s}$ the extension given in (\ref{ti-phi}). When $(g, k)$ is also axisymmetric, the boost solution $(\Om_{\theta}, \ga)$ is invariant under $\ti{\phi}_{s}$ by Lemma \ref{symmetry preserving1}. Now pulling back $\mH(g, k, u)=0$ by $\ti{\phi}_{s}$, we can see that $\phi_{s}^{*}u$ is a solution of $\mH(\phi_{s}^{*}g, \phi_{s}^{*}k, \phi_{s}^{*}u)=0$, hence $\mH(g, k, \phi_{s}^{*}u)=0$. Since $(\Si, e)$ and $(g, k)$ are all invariant under $\phi_{s}$, $ker(L_{0}, \de-\frac{1}{2})$ and hence $ker(L_{0}, \de-\frac{1}{2})^{\perp}$ are also invariant under $\phi_{s}$, which means that $(\phi_{s})^{*}u\in \B_{\rho^{\pr}}\cap ker(L_{0}, \de-\frac{1}{2})^{\perp}$, then uniqueness in Theorem \ref{main theorem first part} implies that $(\phi_{s})^{*}u=u$. So $u$ is axisymmetric, hence is $(g_{u}, k_{u})$ since $\ga$ is also axisymmetric.

For the angular momentum, we have another formular, which is called Komar integral(see section 11.2 in \cite{W} for definition and equivalence with (\ref{angular momentum})),
\begin{equation}
J(S)=\frac{1}{16\pi}\int_{S}*d\xi,
\end{equation}
where $*$ is the Hodge star operator w.r.t $\ga$, and $\xi$ the killing vector field. Since $*d\xi$ is a close form, we know that $J(S)$ is invariant for any two spacelike close surface $S$ and $S^{\pr}$ which are homologous equivalent. So $(\Si, g, k)$ and $(Graph_{u}, g_{u}, k_{u})$ have the same angular momentum.
\end{proof}


\section{Appendix}\label{appendix}

\subsection{Geometry of hypersurface}

Here we show the detailed calculation of the mean curvature of a level surface. Part of the results here already appeared in \cite{Ba1}. First let us calculate the future-directed timelike unit normal vector of $\Sigma_{t}$ defined by $T=-\frac{\nabla t}{|\nabla t|}$, which is given by:
\begin{equation}
T=-\al\nabla t=-\al(\gamma^{tt}\partial_{t}+\gamma^{ti}\partial_{i})=\al^{-1}(\partial_{t}-\be).
\end{equation}
$Graph_{u}$ can be viewed as level surface of $(u-t)=0$, so the unit normal of $Graph_{u}$ is $N=\frac{\nabla(u-t)}{|\nabla(u-t)|}$. Now
\begin{equation}
\begin{split}
\nabla u &=\gamma^{tj}u_{j}\partial_{t}+\gamma^{ij}u_{j}\partial_{i}=\frac{1}{\al^{2}}\lan\be, Du\ran\partial_{t}+Du-\frac{1}{\al^{2}}\lan\be, Du\ran\be^{i}\partial_{i}\\
         &=\frac{1}{\al}\lan\be, Du\ran T+Du.
\end{split}
\end{equation}
So $N$ is calculated as
\begin{equation}
\begin{split}
\nabla(u-t) &=\frac{1}{\al^{2}}\lan\be, Du\ran\partial_{t}+Du-\frac{1}{\al^{2}}\lan\be, Du\ran\be^{i}\partial_{i}+\frac{1}{\al^{2}}\partial_{t}-\frac{1}{\al^{2}}\be^{i}\partial_{i}\\
            &=Du+\frac{1+\lan\be, Du\ran}{\al^{2}}(\partial_{t}-\be)\\
            &=\al^{-1}(1+\lan\be, Du\ran)\big(\frac{\al Du}{1+\lan\be, Du\ran}+T\big).
\end{split}
\end{equation}
Writing $U=\frac{\al Du}{1+\lan\be, Du\ran}$, then $N=\frac{U+T}{|U+T|}$, where $|U+T|=(1-|U|^{2})^{1/2}$, so we get equation (\ref{mean curvature for level surface}).


Denoting $M=Graph_{u}$, let us calculate the mean curvature. For completeness we give the inverse metric matrix $(g_{M})^{-1}$ of $g_{M}$ in (\ref{graph metric}). First we need to calculate the co-frame of (\ref{local frame}). Denoting them by $\al^{i}=a^{i}_{k}dx^{k}+a^{i}_{t}dt:\ i,k=1, 2, 3.$. Then they should satisfy:
\begin{equation}\label{coframe equation}
\al^{i}(\al_{k})=\delta^{i}_{k},\quad \al^{i}(N)=0.
\end{equation}
The last equation gives
\begin{equation}
\begin{split}
(a^{i}_{k}dx^{k}+a^{i}_{t}dt)[\al(U+T)] &=(a^{i}_{k}dx^{k}+a^{i}_{t}dt)(\frac{\al^{2}Du}{1+\lan\be, Du\ran}+(\partial_{t}-\be))\\
                                                                &=a^{i}_{k}(\frac{\al^{2}u^{k}}{1+\lan\be, Du\ran}-\be^{k})+a^{i}_{t}=0.
\end{split}
\end{equation}
So
\begin{equation}
a^{i}_{t}=a^{i}_{k}(\be-\frac{\al^{2}Du}{1+\lan\be, Du\ran})^{k}=a^{i}_{k}(\be-\al U)^{k}.
\end{equation}
Putting into the first on in (\ref{coframe equation}), we have
\begin{equation}
(a^{i}_{k}dx^{k}+a^{i}_{l}(\be^{l}-\al U^{l})dt)(\partial_{k}+u_{k}\partial_{t})=a^{i}_{k}+a^{i}_{l}(\be^{l}-\al U^{l})u_{k}=\delta^{i}_{k}.
\end{equation}
Denoting matrix $A=(a^{i}_{k})$, then the above equations change to the matrix equation
\begin{equation}
A\cdot[id+(\be-\al U)(Du)^{t}]=id.
\end{equation}
Solving the last matrix equation\footnote{The inverse of $Id+\textbf{u}\textbf{v}^{t}$ is given by $Id-\frac{\textbf{u}\textbf{v}^{t}}{1+\textbf{u}\cdot\textbf{v}}$}, we get
\begin{equation}
\begin{split}
a^{i}_{k} &=Id-\frac{(\be^{i}-\al U^{i})u_{k}}{1+\lan\be-\al U, Du\ran}=Id-\frac{(\be^{i}-\al U^{i})u_{k}}{1+\lan\be, Du\ran-(1+\lan\be, Du\ran)|U|^{2}}\\
       &=Id-\nu^{2}\frac{(\be^{i}-\al U^{i})u_{k}}{1+\lan\be, Du\ran}=Id-\nu^{2}(\be/\al-U)^{i}U_{k},
\end{split}
\end{equation}
where we have used $U=\frac{\al Du}{1+\lan\be, Du\ran}$, and $\nu^{-2}=1-|U|^{2}$. Then
\begin{equation}
\begin{split}
a^{i}_{t} &=a^{i}_{k}(\be-\al U)^{k}=(\delta^{i}_{k}-\nu^{2}\frac{(\be^{i}-\al U^{i})u_{k}}{1+\lan\be, Du\ran})(\be^{k}-\al U^{k})\\
          &=\be^{i}-\nu^{2}\frac{(\be^{i}-\al U^{i})\lan\be, Du\ran}{1+\lan\be, Du\ran}-\al U^{i}+\nu^{2}(\be^{i}-\al U^{i})|U|^{2}\\
          &=(1+\nu^{2}|U|^{2})(\be^{i}-\al U^{i})-\nu^{2}\frac{\lan\be, Du\ran}{1+\lan\be, Du\ran}(\be^{i}-\al U^{i})\\
          &=\nu^{2}(\be^{i}-\al U^{i})-\nu^{2}\frac{\lan\be, Du\ran}{1+\lan\be, Du\ran}(\be^{i}-\al U^{i})\\
          &=\frac{\nu^{2}}{1+\lan\be, Du\ran}(\be-\al U)^{i}.
\end{split}
\end{equation}
So the co-frame is given by
\begin{equation}\label{coframe}
\al^{i}=(\delta^{i}_{k}-\nu^{2}(\be/\al-U)^{i}U_{k})dx^{k}+\frac{\nu^{2}}{1+\lan\be, Du\ran}(\be-\al U)^{i}dt.
\end{equation}
Taking inner product of the co-frame with respect to $\gamma^{-1}$, we can calculate $g_{M}^{-1}$.
\begin{equation}\label{inverse graph metric calculation}
\begin{split}
&(g_{M})^{ij}\\
&=\big\lan(\delta^{i}_{k}-\nu^{2}(\be/\al-U)^{i}U_{k})dx^{k}+\frac{\nu^{2}(\be-\al U)^{i}}{1+\lan\be, Du\ran}dt, (\delta^{j}_{l}-\nu^{2}(\be/\al-U)^{j}U_{l})dx^{l}+\frac{\nu^{2}(\be-\al U)^{j}}{1+\lan\be, Du\ran}dt\big\ran_{\gamma}\\
&=(\delta^{i}_{k}-\nu^{2}(\be/\al-U)^{i}U_{k})(\delta^{j}_{l}-\nu^{2}(\be/\al-U)^{j}U_{l})(g^{kl}-\frac{1}{\al^{2}}\be^{k}\be^{l})\\
&+(\delta^{i}_{k}-\nu^{2}(\be/\al-U)^{i}U_{k})\frac{\nu^{2}(\be-\al U)^{j}}{1+\lan\be, Du\ran}\frac{\be^{k}}{\al^{2}}+(\delta^{j}_{k}-\nu^{2}(\be/\al-U)^{j}U_{k})\frac{\nu^{2}(\be-\al U)^{i}}{1+\lan\be, Du\ran}\frac{\be^{k}}{\al^{2}}\\
&-\frac{1}{\al^{2}}\frac{\nu^{4}(\be-\al U)^{i}(\be-\al U)^{j}}{(1+\lan\be, Du\ran)^{2}}\\
&=g^{ij}-\frac{1}{\al^{2}}\be^{i}\be^{j}-\nu^{2}(\be/\al-U)^{j}U^{i}-\nu^{2}(\be/\al-U)^{i}U^{j}+\frac{\nu^{2}(\be-\al U)^{j}\be^{i}\lan\be, Du\ran}{\al^{2}(1+\lan\be, Du\ran)}\\
&+\frac{\nu^{2}(\be-\al U)^{i}\be^{j}\lan\be, Du\ran}{\al^{2}(1+\lan\be, Du\ran)}+\nu^{4}(\be/\al-U)^{i}(\be/\al-U)^{j}|U|^{2}\\
&-\frac{\nu^{4}(\be-\al U)^{i}(\be-\al U)^{j}}{\al^{2}(1+\lan\be, Du\ran)^{2}}\lan\be, Du\ran^{2}+\frac{\nu^{2}}{\al^{2}}\frac{\be^{i}(\be-\al U)^{j}+\be^{j}(\be-\al U)^{i}}{1+\lan\be, Du\ran}\\
&-2\frac{\nu^{4}(\be-\al U)^{i}(\be-\al U)^{j}}{\al^{2}(1+\lan\be, Du\ran)^{2}}\lan\be, Du\ran-\frac{\nu^{4}(\be-\al U)^{i}(\be-\al U)^{j}}{\al^{2}(1+\lan\be, Du\ran)^{2}}\\
&=g^{ij}-\frac{1}{\al^{2}}\be^{i}\be^{j}-\nu^{2}(\be/\al-U)^{j}U^{i}-\nu^{2}(\be/\al-U)^{i}U^{j}+\nu^{4}(\be/\al-U)^{i}(\be/\al-U)^{j}|U|^{2}\\
&+\frac{\nu^{2}}{\al^{2}}[\be^{i}(\be-\al U)^{j}+\be^{j}(\be-\al U)^{i}]-\frac{\nu^{4}}{\al^{2}}(\be-\al U)^{i}(\be-\al U)^{j}\\
&=g^{ij}-\frac{1}{\al^{2}}\be^{i}\be^{j}+\frac{\nu^{2}}{\al^{2}}(\be-\al U)^{i}(\be-\al U)^{j}.
\end{split}
\end{equation}


\subsection{Linear boost estimates on an end}\label{Linear boost estimates}

Here we will give a detailed version of linear boost estimates on an Euclidean end using method in \cite{Ch1} and \cite{ChOM}. It was also mentioned in \cite{BaChrusOM}. We will mainly give the energy estimates needed to prove Theorem \ref{estimates and existence for linear system}. For convenience, we sometime abbreviate $V_{\theta, \la}=V$ in this section. Given a regularly hyperbolic metric $\ga^{\mu\nu}$ and a $\R^{N}$-valued function $u$ in $V_{\theta, \la}$, we can associate it with the \emph{energy-momentum tensor} $T^{\mu\nu}$\footnote{See also equation (4.6) in \cite{Ch1}.}: 
\begin{equation}\label{energy-momentum tensor}
T^{\mu\nu}=G^{\mu\nu\rho\si}D_{\rho}u\cdot D_{\si}u\footnote{Here the inner product of $D_{\rho}u\cdot D_{\si}u=\sum_{k=1}^{N}D_{\rho}u^{k}D_{\si}u^{k}$.},
\end{equation}
where
$$G^{\mu\nu\rho\si}=\ga^{\mu\rho}\ga^{\nu\si}+\ga^{\mu\si}\ga^{\nu\rho}-\ga^{\mu\nu}\ga^{\rho\si}.$$
Given the unit normal $\tilde{n}$ of $\{E_{\tau}\}$ defined in (\ref{future normal}), the \emph{momentum vector field} relative to $\ti{n}$ is
\begin{equation}
P^{\mu}=T^{\mu\nu}\tilde{n}_{\nu}.
\end{equation}
Furthermore, the divergence of $P^{\mu}$ is
\begin{equation}\label{divergence of P}
D_{\mu}P^{\mu}=2(\ga^{\rho\si}\tilde{n}_{\rho}D_{\si}u)\cdot\ga^{\mu\nu}D^{2}_{\mu\nu}u+Q,
\end{equation}
where
$$Q=\La^{\mu\nu}D_{\mu}u\cdot D_{\nu}u,\ \textrm{with}\ \La^{\mu\nu}=D_{\rho}(G^{\mu\nu\rho\si}\tilde{n}_{\si}).$$

Let $N^{-2}=-\lan D\tau, D\tau\ran_{\ga}$ be the lapse function for $\tau$ w.r.t. $\ga$ and $n=ND\tau$ the unit co-normal of $\{E_{\tau}\}$ w.r.t. $\ga$. We introduce an orthonormal frame $\{e_{0}, e_{1}, \cdots, e_{n-1}\}$ w.r.t. $\ga$, such that $e_{0}$ is along the direction of $\tilde{n}^{\mu}=\ga^{\mu\nu}\tilde{n}_{\mu}$, i.e. $e_{0}=\frac{N}{\tilde{N}}\tilde{n}^{\mu}$, where $(\frac{N}{\tilde{N}})^{-2}=|\tilde{n}|_{\ga}^{2}$, and $e_{i}$ perpendicular to $\tilde{n}^{\mu}$. According to Section 2 in \cite{Ch1}, we know that 
$|\tilde{n}|_{\ga}^{2}=\ga^{\mu\nu}\tilde{n}_{\mu}\tilde{n}_{\nu}=(\frac{N}{\tilde{N}})^{-2}$ is bounded from both above and below by some constants depending only on $\theta$ and $h$.

\begin{lemma}\label{P past time-like}
When $\ga$ is regularly hyperbolic, $P^{\mu}$ is past time-like w.r.t $\ga$.
\end{lemma}
\begin{proof}
$T^{\mu\nu}=2D^{\mu}u\cdot D^{\nu}u-|Du|^{2}_{\ga}\ga^{\mu\nu}$, so $P^{\mu}=T^{\mu\nu}\tilde{n}_{\nu}=2D^{\mu}u\cdot D^{\nu}u\tilde{n}_{\nu}-|Du|^{2}_{\ga}\tilde{n}^{\mu}$, and
$$\ga_{\mu\nu}P^{\mu}P^{\nu}=4\ga_{\mu\nu}(D^{\mu}u\cdot D^{\rho}u\tilde{n}_{\rho})(D^{\nu}u\cdot D^{\si}u\tilde{n}_{\si})-4|Du|^{2}_{\ga}(D^{\mu}u\tilde{n}_{\mu}\cdot D^{\nu}u\tilde{n}_{\nu})+|Du|_{\ga}^{4}|\tilde{n}|_{\ga}^{2}$$
$$\leq|Du|_{\ga}^{4}|\tilde{n}|_{\ga}^{2}\leq 0.$$
The first $``\leq"$ comes from Cauchy-Schwartz inequality, and the second comes from the fact that $\tilde{n}$ is time-like w.r.t. $\ga$.

Take $l^{\mu}$ as a future like vector field, then in the orthonormal frame $\{e_{0}, e_{1}, \cdots, e_{n-1}\}$ as above, $l^{0}>\sqrt{\sum_{i=1}^{n-1}(l^{i})^{2}}$, and
$$\ga_{\mu\nu}P^{\mu}l^{\mu}=2[(D_{0}u)l^{0}+(D_{i}u)l^{i}](D_{0}u)(\tilde{N}/N)-[-(D_{0}u)^{2}+\sum(D_{i}u)^{2}](-l_{0})(\tilde{N}/N)$$
$$=[(D_{0}u)^{2}+\sum(D_{i}u)^{2}]l^{0}(\tilde{N}/N)+2D_{i}uD_{0}u l^{i}(\tilde{N}/N)$$
$$\geq (\tilde{N}/N)[(D_{0}u)^{2}+\sum(D_{i}u)^{2}](l^{0}-\sqrt{\sum(l^{i})^{2}})\geq 0.$$
The first $``\geq"$ comes from the Cauchy-Schwartz inequality. So it shows that $P$ is past time-like w.r.t. $\ga$.
\end{proof}

Now we introduce the restriction norm and restriction lemma similar to (\ref{restriction norm2}) and Lemma \ref{restriction}. Given $u\in H_{s, \de}(V_{\theta, \la})$, the restriction norm to hypersurface $E_{\tau}$ is defined as:
\begin{equation}\label{restriction norm}
\|u\|_{H_{s, \de}(E_{\tau}, V_{\theta, \la})}=\big(\Si_{k=0}^{s}\|D^{k}_{t}u|_{E_{\tau}}\|_{H_{s-k, \de+k}(E)}^{2}\big)^{1/2}.
\end{equation}
The following restriction lemma follows similar from Lemma 3.1 in \cite{Ch1}:
\begin{lemma}\label{restriction2}
\emph{(restriction)}. $\forall \tau\in(-\theta, \theta)$, we have the following continuous inclusion:
$$H_{s+1, \de}(V_{\theta, \la})\subset H_{s, \de+\frac{1}{2}}(E_{\tau}, V_{\theta, \la}),$$
for every $s\in\N$ and $\de\in\R$.
\end{lemma}

Now we have the first energy estimates.
\begin{lemma}\label{first energy estimates}
\emph{(First Energy Estimates)}. Assume that $\ga^{\mu\nu}$ is regularly hyperbolic, and $(\ga-\eta)\in C^{\infty}\cap C^{1, 0}(V)$, $a_{1}\in C^{\infty}\cap C^{0, 1}(V)$ and $a_{0}\in C^{\infty}\cap C^{0, 2}(V)$. For $L$ defined in (\ref{weakly coupled linear hyperbolic operator}), with $a_{2}=\ga Id$, every $u\in C^{\infty}_{0}(V)$ satisfies the fundamental energy estimates:
\begin{equation}\label{fundamental energy estimates}
\|u\|_{H_{1, \de+\frac{1}{2}}(E_{\tau}, V)}\leq c(\|u\|_{H_{1, \de+\frac{1}{2}}(E, V)}+\|\be\|_{H_{0, \de+2}(V)}),
\end{equation}
where $0\leq\tau\leq\theta$, $\be=Lu$, and $c$ is a constant depending only on $\theta$, the coefficients $h$ of regular hyperbolicity (\ref{coefficient of regular hyperbolicity}) of $\ga$, and $\|D\ga\|_{C^{0, 1}}+\|a_{1}\|_{C^{0, 1}}+\|a_{2}\|_{C^{0, 2}}$.
\end{lemma}
\begin{proof}
Let $\tilde{P}^{\mu}=\si^{2(\de+\frac{3}{2})}P^{\mu}$. Multiply (\ref{divergence of P}) by $\si^{2(\de+\frac{3}{2})}$, we get:
\begin{equation}\label{weighted divergence of P}
D_{\mu}\tilde{P}^{\mu}=2\si^{2(\de+\frac{3}{2})}(\ga^{\rho\si}\tilde{n}_{\rho}D_{\si}u)\cdot\ga^{\mu\nu}D^{2}_{\mu\nu}u+\tilde{Q},
\end{equation}
where
$$\tilde{Q}=\si^{2(\de+\frac{3}{2})}Q^{\pr},$$
with
$$Q^{\pr}=Q+2(\de+3/2)x^{i}/\si^{2} P^{i}\simeq (D\ga*\ga+\si^{-1}\ga*\ga)Du*Du.$$
Plug in $Lu=\be$,
$$D_{\mu}\tilde{P}^{\mu}=\si^{2(\de+\frac{3}{2})}[2(\ga^{\rho\si}\tilde{n}_{\rho}D_{\si}u)\cdot(\be-a_{1}Du-a_{0}u)+Q^{\pr}].$$

Now we integrate on the upper part $V_{\tau, \la}^{+}=\{x\in V_{\tau, \la}:\ t\geq0\}$ for $\tau\leq\theta$. Since $P$ is compactly supported, the divergence theorem of $(V_{\tau, \la}^{+}, \eta)$ gives,
\begin{equation}\label{integration of divergence of P}
\begin{split}
\int_{E_{\tau}}\tilde{P}^{\mu}\tilde{n}_{\mu}d\Si_{\tau} &-\int_{E} \tilde{P}^{\mu}\tilde{n}_{\mu}d\Si+\int_{L_{\tau, \la}^{+}}\tilde{P}^{\mu}\tilde{\nu}_{\mu}d\si=\int_{V_{\tau, \la}^{+}}\tilde{P}^{\mu}\tilde{n}_{\mu}dx\\
                                       &=\int_{V_{\tau, \la}^{+}}\si^{2(\de+\frac{3}{2})}[2(\ga^{\rho\si}\tilde{n}_{\rho}D_{\si}u)\cdot(\be-a_{1}Du-a_{0}u)+Q^{\pr}]dx,
\end{split}
\end{equation}
where $\tilde{\nu}_{\mu}$ is the unit outer co-normal of the upper lateral boundary $L_{\tau, \la}^{+}=L_{\tau, \la}\cap V_{\tau, \la}^{+}$ under $\eta$, which is future timelike w.r.t. $\ga$ by property $(4)$ of the regular hyperbolicity (\ref{regular hyperbolicity}). Using the fact that $P$ is past time-like(Lemma \ref{P past time-like}), we know that
$$\tilde{P}^{\mu}\tilde{\nu}_{\mu}=\si^{2(\de+\frac{3}{2})}P^{\mu}\tilde{\nu}_{\mu}\geq 0,\ \textrm{on}\ L_{\tau, \la}^{+}.$$

Now define:
\begin{equation}\label{x1}
x_{1}(\tau)=\int_{E_{\tau}}|\si^{\de+3/2}Du|^{2}d\Si=\|Du\|^{2}_{H_{0, \de+\frac{3}{2}}(E_{\tau}, V)}.
\end{equation}
Since $\tilde{n}_{\mu}=\tilde{N}D_{\mu}\tau=\frac{\tilde{N}}{N}n_{\mu}$,
$$P^{\mu}\tilde{n}_{\mu}=T^{\mu\nu}\tilde{n}_{\mu}\tilde{n}_{\nu}=2(\ga^{\mu\si}D_{\si}u\tilde{n}_{\mu})^{2}-|Du|_{\ga}^{2}|\tilde{n}|_{\ga}^{2}$$
$$=(\frac{\tilde{N}}{N})^{2}(2n^{\mu}n^{\nu}+\ga^{\mu\nu})D_{\mu}u D_{\nu}u.$$
Using Proposition 2.3 in \cite{Ch1}, $\Ga^{\mu\nu}=2 n^{\mu}n^{\nu}+\ga^{\mu\nu}$ is uniformly elliptic, with the elliptic coefficient depending only on the coefficient of regular hyperbolicity $h$. Using equation (2.8)(2.13) of \cite{Ch1}, $d\Si_{\tau}\simeq c d\Si$, with $c$ depending only on $\theta$, so we have:
$$\int_{E_{\tau}}\tilde{P}^{\mu}\tilde{n}_{\mu}d\Si_{\tau}\geq c_{1}^{-1}x_{1}(\tau),$$
$$\int_{E}\tilde{P}^{\mu}\tilde{n}_{\mu}d\Si_{0}\leq c_{1}x_{1}(0),$$
where $c_{1}$ is a constant depending only on $\theta$ and the regular hyperbolicity coefficient $h$. Now using Cauchy-Schwartz inequality and the fact $dx=\si d\tau d\Si$ to the right hand side of  (\ref{integration of divergence of P}),
$$|\int_{V_{\tau, \la}^{+}}2\si^{2(\de+\frac{3}{2})}(\ga^{\rho\si}\tilde{n}_{\rho}D_{\si}u)\cdot\be dx|\leq c_{1}\int_{0}^{\tau}\|Du\|_{H_{0, \de+\frac{3}{2}}(E_{\tau^{\pr}}, V)}\|\be\|_{H_{0, \de+\frac{5}{2}}(E_{\tau^{\pr}}, V)}d\tau^{\pr};$$
$$|\int_{V_{\tau, \la}^{+}}2\si^{2(\de+\frac{3}{2})}(\ga^{\rho\si}\tilde{n}_{\rho}D_{\si}u)\cdot a_{1}Du dx|\leq c_{1}\|a_{1}\|_{C^{0, 1}}\int_{0}^{\tau}\|Du\|^{2}_{H_{0, \de+\frac{3}{2}}(E_{\tau^{\pr}}, V)}d\tau^{\pr};$$
$$|\int_{V_{\tau, \la}^{+}}2\si^{2(\de+\frac{3}{2})}(\ga^{\rho\si}\tilde{n}_{\rho}D_{\si}u)\cdot a_{0}u dx|\leq c_{1}\|a_{0}\|_{C^{0, 2}}\int_{0}^{\tau}\|Du\|_{H_{0, \de+\frac{3}{2}}(E_{\tau^{\pr}}, V)}\|u\|_{H_{0, \de+\frac{1}{2}}(E_{\tau^{\pr}}, V)}d\tau^{\pr};$$
$$|\int_{V_{\tau, \la}^{+}}2\si^{2(\de+\frac{3}{2})}Q^{\pr} dx|\leq c_{1}(1+\|D\ga\|_{C^{0, 1}})\int_{0}^{\tau}\|Du\|^{2}_{H_{0, \de+\frac{3}{2}}(E_{\tau^{\pr}}, V)}d\tau^{\pr},$$
where $c_{1}$ denotes a constant depending only on the regular hyperbolicity coefficient $h$. Now define:
\begin{equation}\label{x0}
x_{0}(\tau)=\int_{E_{\tau}}|\si^{\de+1/2}u|^{2}d\Si=\|u\|^{2}_{H_{0, \de+\frac{1}{2}}(E_{\tau}, V)},
\end{equation}
then (\ref{integration of divergence of P}) can be changed to
\begin{equation}\label{x1 estimates}
x_{1}(\tau)\leq c_{2}\big\{x_{1}(0)+\int_{0}^{\tau}\|\be\|_{H_{0, \de+\frac{5}{2}}(E_{\tau^{\pr}}, V)}x_{1}(\tau^{\pr})^{1/2}d\tau^{\pr}+m_{1}\int_{0}^{\tau}y_{1}(\tau^{\pr})d\tau^{\pr}\big\},
\end{equation}
where $c_{2}$ is a constant depending only on $\theta$ and the regular hyperbolicity coefficient $h$, and
\begin{equation}\label{m1}
m_{1}=\|D\ga\|_{C^{0, 1}}+\|a_{1}\|_{C^{0, 1}}+\|a_{0}\|_{C^{0, 2}},
\end{equation}
\begin{equation}\label{y1}
y_{1}(\tau)=x_{1}(\tau)+x_{0}(\tau)=\|u\|^{2}_{H_{1, \de+\frac{1}{2}}(E_{\tau}, V)}.
\end{equation}

Using Cauchy-Schwartz inequality,
$$(u(\tau)-u(0))^{2}=(\int_{0}^{\tau}\frac{\partial u}{\partial\tau^{\pr}}d\tau^{\pr})^{2}\leq\tau\int_{0}^{\tau}(\frac{\partial u}{\partial\tau^{\pr}})^{2}d\tau^{\pr}.$$
Consider the projection map $\pi: V_{\theta, \la}\rightarrow E$ defined by $\pi(\bar{x}, t)=\bar{x}$, then $E^{\pr}_{\tau}=\pi(E_{\tau})\subset E^{\pr}_{\tau^{\pr}}$ if $\tau^{\pr}<\tau$, then
$$\int_{E^{\pr}_{\tau}}|\si^{\de+1/2}(u(\tau)- u(0))|^{2}d\Si\leq\tau\int_{0}^{\tau}\big\{\int_{E^{\pr}_{\tau}}|\si^{\de+3/2}\frac{\partial u}{\partial t}|^{2}d\Si\big\}d\tau^{\pr}\leq \tau\int_{0}^{\tau}x_{1}(\tau^{\pr})d\tau^{\pr}.$$
So,
\begin{equation}\label{x0 estimates}
x_{0}(\tau)\leq 2x_{0}(0)+2\tau\int_{0}^{\tau}x_{1}(\tau^{\pr})d\tau^{\pr}.
\end{equation}
Adding (\ref{x1 estimates}) and (\ref{x0 estimates}), we can get the integral inequality,
\begin{equation}
y_{1}(\tau)\leq c_{2}\big\{y_{1}(0)+\int_{0}^{\tau}\|\be\|_{H_{0, \de+\frac{5}{2}}(E_{\tau^{\pr}}, V)}y_{1}^{1/2}(\tau^{\pr})d\tau^{\pr}+m_{1}\int_{0}^{\tau}y_{1}(\tau^{\pr})d\tau^{\pr}\big\}
\end{equation}
Using the Gronwall lemma,
\begin{equation}
y_{1}^{1/2}(\tau)\leq exp(\frac{1}{2}c_{2}m_{1}\tau)\big\{y_{1}^{1/2}(0)+\frac{1}{2}\int_{0}^{\tau}e^{\frac{1}{2}c_{2}m_{1}\tau^{\pr}}c_{2}\|\be\|_{H_{0, \de+\frac{5}{2}}(E_{\tau^{\pr}}, V)}d\tau^{\pr}\big\}.
\end{equation}
Hence we finished the proof by using $y_{1}^{1/2}(\tau)=\|u\|_{H_{1, \de+\frac{1}{2}}(E_{\tau}, V)}$.
\end{proof}

This result can be weaken to the case of rough coefficients by approximation methods.
\begin{lemma}\label{fundamental energy estimates2}
If $\ga$ is regularly hyperbolic on $V$, $(\ga-\eta)\in C^{1, 0}(V)$, $a_{1}\in C^{0, 1}(V)$ and $a_{0}\in C^{0, 2}(V)$, then every $u\in H_{2, \de}(V)$ satisfies the fundamental energy estimates (\ref{fundamental energy estimates}), with $\be=Lu$.
\end{lemma}
\begin{proof}
This comes from an approximation argument exactly the same as Lemma 4.2 in \cite{Ch1}.
\end{proof}

Using more differentiability of the coefficients, we can improve the energy estimates containing high order derivatives. 
\begin{lemma}
\emph{(High Order Estimates)}. Given $s\leq s^{\pr}$ with $s^{\pr}$ defined in (\ref{s prime}). If $\ga$ is regularly hyperbolic, $(\ga-\eta)\in C^{\infty}(V)$, $a_{1}\in C^{\infty}(V)$ and $a_{0}\in C^{\infty}(V)$, then every $u\in C_{0}^{\infty}(V)$ satisfies the main energy estimates:
\begin{equation}\label{high order energy estimates}
\|u\|_{H_{s, \de+\frac{1}{2}}(E_{\tau}, V)}\leq c(\|u\|_{H_{s, \de+\frac{1}{2}}(E, V)}+\|\be\|_{H_{s-1, \de+2}(V)}),
\end{equation}
where $0\leq\tau\leq\theta$, $\be=Lu$, and $c$ is a constant depending only on $\theta$, the coefficient of regular hyperbolicity $h$ and $m$(defined in (\ref{coefficient m})).
\end{lemma}
\begin{proof}
Apply $D^{i-1}$ for $2\leq i\leq s$ to $Lu=\be$, we can get
\begin{equation}\label{equation  of ui}
\ga^{\mu\nu}D^{2}_{\mu\nu}u^{[i-1]}=\be^{[i-1]},
\end{equation}
where $u^{[i-1]}=D^{i-1}u$, and
$$\be^{[i-1]}=D^{i-1}\be-\sum_{p=1}^{i-1}\binom{i-1}{p}D^{p}\ga D^{i+1-p}u-\sum_{p=0}^{i-1}\binom{i-1}{p}(D^{p}a_{1}D^{i-p}u+D^{p}a_{0}D^{i-1-p}u).$$
Now define
\begin{equation}\label{xi}
x_{i}(\tau)=\int_{E_{\tau}}|\si^{\de+i+\frac{1}{2}}D^{i}u|^{2}d\Si=\|D^{i}u\|_{H_{0, \de+i+1/2}(E_{\tau}, V)}^{2},
\end{equation}
and apply (\ref{x1 estimates}) in Lemma \ref{first energy estimates} to (\ref{equation  of ui}) with $\de$ replaced by $\de+i-1$, then
$$x_{i}(\tau)\leq c_{1}\big\{x_{i}(0)+\int_{0}^{\tau}\|\be^{[i-1]}\|_{H_{0, \de+i+3/2}(E_{\tau^{\pr}, V})}x_{i}^{1/2}(\tau^{\pr}) d\tau^{\pr}+m_{1}\int_{0}^{\tau}x_{i}(\tau^{\pr})d\tau^{\pr}\big\},$$
with $c_{1}$ depending only on the coefficient of regular hyperbolicity $h$ and $m_{1}$ defined in (\ref{m1}). Compared to (\ref{x1 estimates}), we have only $x_{i}(\tau^{\pr})$ in the third term since there is no first order term $Óa_{0}Ó$ in (\ref{equation  of ui}). Now using the multiplication lemma \ref{properties of weighted sobolev space} and restriction lemma \ref{restriction2} in the case
$$H_{s_{2}-p-1, \de_{2}+p+1/2}(E_{\tau}, V)\times H_{p-1, \de+i+3/2-p}(E_{\tau}, V)\rightarrow H_{0, \de+i+3/2}(E_{\tau}, V),$$
we get 
$$\|D^{p}\ga D^{i+1-p}u\|_{H_{0, \de+i+3/2}(E_{\tau}, V)}\leq c_{3}\|D\ga\|_{H_{s_{2}-1, \de_{2}+1}(V)}\|Du\|_{H_{i-1, \de+3/2}(E_{\tau}, V)},$$
with $c_{3}$ a constant depending only on $i$ and $\de$. Similarly,
$$\|D^{p}a_{1}D^{i-p}u+D^{p}a_{0}D^{i-1-p}u\|_{H_{\tau}(E_{\tau}, V)}\leq c_{3}(\|a_{1}\|_{H_{s_{1}, \de_{1}}(V)}+\|a_{0}\|_{H_{s_{0}, \de_{0}}(V)})\|u\|_{H_{i, \de+1/2}(E_{\tau}, V)}.$$
So
$$\|\be^{[i-1]}\|_{H_{0, \de+i+\frac{3}{2}}(E_{\tau^{\pr}}, V)}\leq \|D^{i-1}\be\|_{H_{0, \de+i+\frac{3}{2}}(E_{\tau^{\pr}}, V)}+c_{4}m\|u\|_{H_{i, \de+1/2}(E_{\tau}, V)},$$
where $c_{4}$ is a constant depending on $s$, $\de$, and $m$ is given by (\ref{coefficient m}). Now define:
\begin{equation}\label{yi}
y_{i}(\tau)=y_{1}(\tau)+\sum_{j=2}^{i}x_{j}(\tau)=\|u\|_{H_{i, \de+1/2}(E_{\tau}, V)}^{2}.
\end{equation}
We have
$$x_{i}(\tau)\leq c_{1}\big\{x_{i}(0)+\int_{0}^{\tau}\|D^{i-1}\be\|_{H_{0, \de+i+\frac{3}{2}}(E_{\tau^{\pr}}, V)}x_{i}^{1/2}(\tau^{\pr})d\tau^{\pr}+c_{4}(m+m_{1})\int_{0}^{\tau}y_{i}(\tau^{\pr})d\tau^{\pr}\big\}.$$
Summing all $i$ from $1$, we can get
\begin{equation}
y_{i}(\tau)\leq c_{1}\big\{y_{i}(0)+\int_{0}^{\tau}\|\be\|_{H_{i-1, \de+5/2}(E_{\tau^{\pr}}, V)}y_{i}^{1/2}(\tau^{\pr})d\tau^{\pr}+c_{4}(m+m_{1})\int_{0}^{\tau}y_{i}(\tau^{\pr})d\tau^{\pr}\big\}.
\end{equation}
Using the Gronwall lemma,
\begin{equation}
y_{i}^{1/2}(\tau)\leq\exp(c_{5}(m+m_{1})\tau)\big\{y_{i}^{1/2}(0)+c_{1}\int_{0}^{\tau}e^{c_{5}(m+m_{1})\tau^{\pr}}|\be\|_{H_{i-1, \de+5/2}(E_{\tau^{\pr}}, V)}d\tau^{\pr}\big\},
\end{equation}
where $c_{5}=\frac{1}{2}c_{1}c_{4}$. Hence we finish the proof realizing $m_{1}\leq c_{6}m$ by the imbedding lemma \ref{properties of weighted sobolev space}.
\end{proof}

Using the equation $Lu=\be$ and an argument similar to Lemma 4.4 in \cite{Ch1}, we can estimate $\|u\|_{H_{s, \de+1/2}(E, V)}$ by the spatial norms $\|\phi\|_{H_{s, \de+1/2}(E)}$, $\|\psi\|_{H_{s-1, \de+3/2}(E)}$ and $\|\be\|_{H_{s-2, \de+5/2}(E, V)}$, where $\phi=u|_{E}$ and $\psi=D_{t}u|_{E}$. We need the following technical lemma which says that we can take the division in the Banach algebra $H_{s, \de}(U)$, when $s>\frac{n}{2}$ and $\de>-\frac{n}{2}$.
\begin{lemma}\label{reciprocal}
Given $U$ satisfied the extended cone property, $s>\frac{n}{2}$, $\de>-\frac{n}{2}$ and a function $f$, if $(f-1)\in H_{s, \de}(U)$, and $|f|\geq c>0$, then $(f^{-1}-1)\in H_{s, \de}(U)$, furthermore, $\|f^{-1}-1\|_{H_{s, \de}(U)}$ is bounded by a constant depending only on $n$, $s$, $\de$ and $\|f-1\|_{H_{s, \de}(U)}$.
\end{lemma}
\begin{proof}
Since $|f|\geq c>0$, $f^{-1}$ is well defined. Since $f^{-1}-1=-\frac{f-1}{f}$ and $|f|^{-1}\leq c^{-1}$ uniformly bounded, $(f^{-1}-1)\in H_{0, \de}(U)$. Now $D^{\al}(f^{-1}-1)=\sum_{\al_{1}+\cdots +\al_{l}=\al}\frac{D^{\al_{1}}f\cdots D^{\al_{l}}f}{f^{|\al|+1}}$, where $\al$ is multi-indexes, with $1\leq |\al|\leq s$. Since $(f^{|\al|})^{-1}$ is uniformly bounded, and using the multiplication Lemma \ref{properties of weighted sobolev space}, $D^{\al_{1}}f\cdots D^{\al_{l}}f\in H_{0, \de+|\al|}(U)$, hence $D^{\al}(f^{-1}-1)\in H_{0, \de+|\al|}(U)$. So $(f^{-1}-1)\in H_{s, \de}(U)$. The norm bounds follows from the norm bounds of each $D^{\al}(f^{-1}-1)$.
\end{proof}

\begin{lemma}
Given an operator $L$ defined in (\ref{weakly coupled linear hyperbolic operator}) satisfying Hypothesis $(1)$ and $(2)$, then every $u\in H_{s+1, \de}(V)$ with $2\leq s\leq s^{\pr}$, which solves $Lu=\be$ satisfies:
\begin{equation}\label{estimates of u on E}
\|u\|_{H_{s, \de+1/2}(E, V)}\leq c(\|\phi\|_{H_{s, \de+1/2}(E)}+\|\psi\|_{H_{s-1, \de+3/2}(E)}+\|\be\|_{H_{s-2, \de+5/2}(E, V)}),
\end{equation}
where $\phi=u|_{E}$, $\psi=D_{t}u|_{E}$ and $c$ is a constant depending only on $s$, $\de$ and $\mu$(defined in (\ref{coefficient mu})).
\end{lemma}
\begin{proof}
By the restriction Lemma \ref{restriction2}, $u\in H_{s+1, \de}(V)$ implies that $\phi\in H_{s, \de+1/2}(E)$ and $\psi\in H_{s-1, \de+3/2}(E)$. Now define the following functions on $E$:
$$\psi^{[p]}=D^{p}_{t}u,\quad 0\leq p\leq s.$$
Since
$$\|u\|_{H_{s, \de+1/2}(E, V)}^{2}=\sum_{p=0}^{s}\|\psi^{[p]}\|_{H_{s-p, \de+p+1/2}(E)}^{2},$$
we only need to prove that:
$$\|\psi^{[p]}\|_{H_{s-p, \de+p+1/2}(E)}\leq c_{p}(\|\phi\|_{H_{s, \de+1/2}(E)}+\|\psi\|_{H_{s-1, \de+3/2}}+\|\be\|_{H_{s-2, \de+5/2}(E, V)}).$$
It is true for $p=0, 1$. Let us use a reduction argument to prove this for all $p\leq s$. Suppose it is true for $0\leq q\leq p-1$.  Take $D_{t}^{p-2}$ to the equation $Lu=\be$, and move all the terms containing $t$-derivatives of $u$ of order less than $p$, i.e. $D_{t}^{q}u$ with $q<p$, to the right hand side, then we get
\begin{equation}\label{psi p}
\begin{split}
\ga^{00} & \psi^{[p]}=D_{t}^{p-2}\be-\sum_{q=0}^{p-3}\binom{p-2}{q}(D_{t}^{p-2-q}\ga^{00})\psi^{[q+2]}-\sum_{q=0}^{p-2}\binom{p-2}{q}\big\{2(D_{t}^{p-2-q}\ga^{0i})D_{i}\psi^{[q+1]}\\
              & +(D_{t}^{p-2-q}\ga^{ij})D_{i}D_{j}\psi^{[q]}+(D_{t}^{p-2-q}a_{1}^{0})\psi^{[q+1]}+(D_{t}^{p-2-q}a_{1}^{i})D_{i}\psi^{[q]}+(D_{t}^{p-2-q}a_{0})\psi^{[q]}\big\}.
\end{split}
\end{equation}
Using the multiplication Lemma \ref{properties of weighted sobolev space} and Hypothesis $(1)$ in the case:
$$H_{s_{2}-1-(p-2-q), \de_{2}+1/2+(p-2-q)}(E)\times H_{s-(q+2), \de+1/2+(q+2)}(E)\rightarrow H_{s-p, \de+1/2+p}(E),$$
we can estimate
$$\|(D_{t}^{p-2-q}\ga^{0i})D_{i}\psi^{[q+1]}\|_{H_{s-p, \de+1/2+p}(E)}\leq c_{3}\|\ga-\eta\|_{H_{s_{2}-1, \de_{2}+1/2}(E, V)}\|\psi^{[q+1]}\|_{H_{s-(q+2), \de+1/2+(q+2)}(E)},$$
where $c_{3}$ is a constant depending only on $s$ and $\de$. Now using similar arguments to evaluate the $H_{s-p, \de+p+1/2}(E)$ norm of other terms in (\ref{psi p}), together with our inductive hypothesis, we can get
\begin{equation}\label{ga00 psi p}
\begin{split}
\|\ga^{00}\psi^{[p]}\|_{H_{s-p, \de+p+1/2}(E)} &\leq \|D_{t}^{p-2}\be\|_{H_{s-p, \de+1/2+p}(E)}+ c_{4}\mu\sum_{q=0}^{p-1}\|\psi^{[q]}\|_{H_{s-q, \de+1/2+q}(E)}\\
&\leq c_{p}^{\pr}(\|\phi\|_{H_{s, \de+1/2}(E)}+\|\psi\|_{H_{s-1, \de+3/2}(E)}+\|\be\|_{H_{s-2, \de+5/2}(E, V)}),
\end{split}
\end{equation}
where $\mu$ is defined in (\ref{coefficient mu}), $c_{4}$ is a constant depending only on $s$, $p$ and $\de$, while $c_{p}^{\pr}$ a constant depending only on $\mu$, $s$, $p$ and $\de$.

Here
$$\ga^{00}|_{E}=(\ga^{\mu\nu}D_{\mu}t D_{\nu}t)|_{t=0}=\si^{2}(\ga^{\mu\nu}D_{\mu}\tau D_{\nu}\tau)|_{t=0}=-N^{-2}\si^{2}\leq -c<0,$$
where $c>0$ is a constant depending only on $\theta$ and $h$ according to Section 2 in \cite{Ch1}. Now $(\ga-\eta)\in H_{s_{2}, \de_{2}}(V)$ implies that $(\ga^{00}+1)|_{E}\in H_{s_{2}-1, \de_{2}+1/2}(E)$, hence $((\ga^{00})^{-1}+1)\in H_{s_{2}-1, \de_{2}+1/2}(E)$ by Lemma \ref{reciprocal}, and furthermore $\|(\ga^{00})^{-1}+1\|_{H_{s_{2}-1, \de_{2}+1/2}(E)}$ is bounded by a constant depending only on $n$, $s_{2}$, $\de_{2}$ and $\|\ga^{00}+1\|_{H_{s_{2}-1, \de_{2}+1/2}(E)}$.
Now multiply $\ga^{00}\psi^{[p]}$ by $(\ga^{00})^{-1}$, and use equation (\ref{ga00 psi p}) and the multiplication Lemma \ref{properties of weighted sobolev space}, then we finish the proof.
\end{proof}

By combining all the above estimates, we can get the energy estimates in Theorem \ref{estimates and existence for linear system}.
\begin{theorem}\label{boost estimates for linear system}
Given $L$ a differential operator defined by (\ref{weakly coupled linear hyperbolic operator}) in $V_{\theta, \la}$, satisfying hypotheses (1) and (2). Let $\be\in H_{s-1, \de+2}(V_{\theta, \la})$, $\phi\in H_{s, \de+\frac{1}{2}}(E)$ and $\psi\in H_{s-1, \de+\frac{3}{2}}(E)$, with $2\leq s\leq s^{\pr}$, $\de\in\R$. Then every $u\in H_{s+1, \de}(V)$, which solves $Lu=\be$, with $u|_{E}=\phi,\ D_{t}u|_{E}=\psi$ satisfies the estimates (\ref{estimates for linear equation}).
\end{theorem}
\begin{proof}
First we can plug in (\ref{estimates of u on E}) to $(\ref{high order energy estimates})$. Then it follows from an approximation argument similar to the proof of Lemma 4.5 in \cite{Ch1} and an integration of (\ref{high order energy estimates}) w.r.t. $\tau$ on $[-\theta, \theta]$.
\end{proof}



\parindent 0ex
Department of Mathematics, Stanford University\\
Stanford, CA 94305\\
E-mail: xzhou08@math.stanford.edu

\end{document}